\documentclass[11pt,a4paper]{amsart}
\usepackage{graphics}
\usepackage{amscd}
\usepackage{comment}
\usepackage{oldgerm}
\usepackage{amsmath}
\usepackage{amssymb}
\usepackage{amsthm}
\newtheorem{thm}{Theorem}[section]
\newtheorem{prop}[thm]{Proposition}
\newtheorem{lem}[thm]{Lemma}

\newtheorem{definition}[thm]{Definition}
\newtheorem{rem}[thm]{Remark}

\newcommand{\Z}{\mathbb{Z}}
\newcommand{\N}{\mathbb{N}}
\newcommand{\R}{\mathbb{R}}
\newcommand{\C}{\mathbb{C}}

\newcommand{\del}{\partial}
\newcommand{\delb}{\overline{\partial}}

\begin{document}
\title[The geometric quantizations]{The geometric quantizations and 
the measured Gromov-Hausdorff convergences}
\author{Kota Hattori}
\address{Keio University, 
3-14-1 Hiyoshi, Kohoku, Yokohama 223-8522, Japan}
\email{hattori@math.keio.ac.jp}
\thanks{Partially 
supported by Grant-in-Aid for Young Scientists (B) Grant Number
16K17598 and by Grant-in-Aid for 
Scientific Research (C) 
Grant Number 19K03474.}
\subjclass[2010]{53C07, 58J50}
\maketitle
\begin{abstract}
On a compact symplectic manifold 
$(X,\omega)$ with a prequantum line bundle
$(L,\nabla,h)$, 
we consider the 
one-parameter family of 
$\omega$-compatible 
complex structures which converges to the 
real polarization coming from the 
Lagrangian torus fibration. 
There are several researches which show 
that the holomorphic sections of the line bundle 
localize at Bohr-Sommerfeld fibers. 
In this article we consider the one-parameter 
family of the Riemannian metrics on the 
frame bundle of $L$ determined by the complex structures
and $\nabla,h$, and we can see that their diameters 
diverge. 
If we fix a base point in some 
fibers of the Lagrangian fibration we 
can show that they measured 
Gromov-Hausdorff converge to 
some pointed metric measure spaces 
with the isometric $S^1$-actions, which may depend 
on the choice of the base point. 
We observe that 
the properties of the $S^1$-actions 
on the limit spaces actually depend on 
whether the base point is in 
the Bohr-Sommerfeld fibers 
or not. 
\end{abstract}
\section{Introduction}
In this article we introduce a new approach 
to the geometric quantization from the viewpoint 
of the convergence of the 
Riemannian manifolds with respect to 
the measured Gromov-Hausdorff 
topology. 
On a compact symplectic manifold 
$(X,\omega)$ of dimension $2n$, 
a prequantum line bundle $(L,\nabla,h)$ 
is a triple of a complex line bundle 
$L$, hermitian metric $h$ and 
connection $\nabla$ preserving $h$ 
with curvature form $F^\nabla=-\sqrt{-1}\omega$. 
By considering the following 
geometric structures compatible with $\omega$, 
we can equip $L$ 
with the finite dimensional vector subspace
consisting of the special sections of $L$. 
The first one is an 
$\omega$-compatible complex structure 
$J$, then denote by 
\begin{align*}
H^0(X_J,L)
=\{ s\in C^\infty(L);\, \nabla_{\delb_J} s=0\}
\end{align*}
the space of all of the holomorphic sections 
of $L$. 

The second one is a Lagrangian fibration 
$\mu\colon X\to Y$, where 
$Y$ is a smooth manifold, 
all of the points in $Y$ 
are regular values of $\mu$, 
and all fibers are compact and connected. 
Then put 
\begin{align*}
V_\mu
&:= \bigoplus_{y\in Y}H^0\left( \mu^{-1}(y),
L|_{\mu^{-1}(y)},\nabla|_{\mu^{-1}(y)}\right),\\
H^0\left( \mu^{-1}(y),
L|_{\mu^{-1}(y)},\nabla|_{\mu^{-1}(y)}\right)
&:= 
\left\{ s\in C^\infty(L|_{\mu^{-1}(y)});\,
\nabla|_{\mu^{-1}(y)} s\equiv 0\right\}.
\end{align*}
$\mu^{-1}(y)$ is called a 
{\it Bohr-Sommerfeld fiber} if 
$L|_{\mu^{-1}(y)}$ 
has nontrivial parallel sections. 
Tyurin showed in \cite[Proposition 3.2]{Tyurin2002} 
that if $X$ is compact then 
there are at most finitely many 
Bohr-Sommerfeld fibers, accordingly, 
$\dim V_\mu$ is finite. 

In many examples of symplectic manifolds 
with some complex structures and 
Lagrangian fibrations, it is observed that 
\begin{align*}
\dim V_\mu = \dim H^0(X_J,L)
\end{align*}
when the Kodaira vanishing holds, 
which can be interpreted as the localization of 
the Riemann-Roch index to the Bohr-Sommerfeld fibers, 
and discussed by Andersen \cite{andersen1997}, 
by Fujita, Furuta and Yoshida \cite{FFY2010}, and 
by Kubota \cite{Kubota2016}. 

Moreover, on smooth toric varieties, 
Baier, Florentino, Mour\~{a}o and Nunes 
have constructed a one parameter family 
of the pairs of the complex structures 
and the basis of the spaces of 
holomorphic sections of $L$, 
then showed that the holomorphic sections localize on 
the Bohr-Sommerfeld fibers in \cite{BFMN2011}. 
The similar phenomena were observed in 
the case of the abelian varieties by Baier, Mour\~{a}o and Nunes \cite{BMN2010} 
and the flag varieties by Hamilton and Konno \cite{HamiltonKonno2014}.
In these examples, 
the family of complex structures and 
holomorphic sections are described concretely.

In the context of the geometric quantization, 
the $\omega$-compatible complex structures 
and Lagrangian fibrations are treated uniformly 
by using the notion of polarizations. 
The one-parameter families of 
complex structures given in the above papers 
are taken such that the 
K\"ahler polarizations corresponding to them 
converge to the 
real polarization corresponding to the Lagrangian fibration. 

Recently, Yoshida showed the 
localization of holomorphic sections of prequantum line bundle 
to the Bohr-Sommerfeld fiber 
if $X$ admits a Lagrangian fibration with 
a complete base  
in \cite{yoshida2019adiabatic}, 
where the family of complex structures 
are taken such that it converges to 
the real polarization corresponding to the 
Lagrangian fibration.

In this article, 
we also study the behavior 
of holomorphic sections of $L$ 
where the family of complex structures 
converges to the real polarization 
from the view of the point of the 
measured Gromov-Hausdorff convergence. 
Fix an $\omega$-compatible 
complex structure $J$. 
Then $H^0(X_J,L)$ can be identified with 
the eigenspace of a Laplace operator as follows. 
Put $S:=\{ u\in L;\,h(u,u)=1\}\subset L$ be 
the orthogonal frame bundle of $(L,h)$, 
then there is the standard identification 
\begin{align*}
C^\infty(X,L) \cong 
(C^\infty(S)\otimes \C)^\rho,
\end{align*}
where $\rho\colon S^1\to GL(\C)$ 
is the $1$-dimensional unitary 
representation of $S^1$ 
defined by $\rho(e^{\sqrt{-1}t}):=e^{\sqrt{-1}t}$, 
and the $S^1$-action on 
$C^\infty(S)\otimes \C$ is defined by 
\begin{align*}
\{ e^{\sqrt{-1}t}\cdot(f\otimes \xi)\} (u)
:= e^{-\sqrt{-1}t}f(ue^{\sqrt{-1}t})\otimes \xi
\end{align*}
for any $f\in C^\infty(S)$, $\xi\in\C$ and 
$u\in S$. 
The connection $\nabla$ gives the 
connection $1$-form on $S$ and 
the decomposition of $TS$ into 
the horizontal and vertical subspaces. 
Then we have the Riemannian metric $\hat{g}$ 
on $S$ which respects the connection form 
and 
the K\"ahler metric $g_J:=\omega(\cdot,J\cdot)$. 
The precise definition of $\hat{g}$ is 
given by Section \ref{principal metric}. 
Denote by $\Delta^{\hat{g}}$ the Laplace operator 
of $\hat{g}$. Since $S^1$ acts on $(S,\hat{g})$ 
isometrically, 
the $\C$-linear extension of 
$\Delta_{\hat{g}}$ gives the operator acting 
on $(C^\infty(S)\otimes \C)^\rho$. 
Then we can see 
that $H^0(X_J,L)$ is identified with 
the eigenspace of 
\begin{align*}
\Delta_{\hat{g}}\colon (C^\infty(S)\otimes \C)^\rho
\to (C^\infty(S)\otimes \C)^\rho
\end{align*}
associate with the eigenvalue $n+1$.

Now, we suppose that a one-parameter 
family of the $\omega$-compatible 
complex structures $\{ J_s\}_{s>0}$
on $X$ is given, 
then we consider the one-parameter 
family of the operators 
\begin{align*}
\Delta_{\hat{g}_s}\colon (C^\infty(S)\otimes \C)^\rho
\to (C^\infty(S)\otimes \C)^\rho.
\end{align*}
There are several research of the spectral convergence of 
the metric Laplacian on Riemannian manifolds 
or the connection Laplacians on vector bundles under the 
convergence of the spaces in the sense of the 
measured Gromov-Hausdorff topology 
\cite{Cheeger-Colding3}\cite{Fukaya1987}\cite{Kasue2011}\cite{KuwaeShioya2003}\cite{Lott-Dirac2002}\cite{Lott-form2002}. 
Therefore, there should be 
the significant relation between 
the convergence of principal bundle 
$S$ with the connection metric $\hat{g}_s$ 
and the convergence of holomorphic sections 
with respect to $J_s$. 
This article focus on the 
convergence of $(S,\hat{g}_s,p)$ as $s\to 0$ 
in the sense of the pointed 
$S^1$-equivariant measured Gromov-Hausdorff 
topology and we study the metric measure spaces 
appearing as the limit.

Now we explain the main result of this article. 
Let $(X,\omega)$ be a symplectic manifold 
of dimension $2n$, which is not necessarily to be 
compact, 
$(L,\nabla,h)$ be a prequantum line bundle 
and $\{ J_s\}_{0<s\le 1}$ be 
a smooth family of $\omega$-compatible 
complex structures. 
Assume that there is a smooth map 
$\mu \colon X\to Y$, where $Y$ is an $n$-dimensional 
smooth manifold, and for any regular values 
$y$ of $\mu$, $\mu^{-1}(y)$ is 
a compact connected Lagrangian submanifold. 
Fix a regular value $y$. 
We assume that $\{ J_s\}_{0<s\le 1}$ 
converges to the real polarizations 
induced by $\mu$ near 
$\mu^{-1}(y)$ as $s\to 0$, 
there is a constant $\kappa\in\R$ such that 
${\rm Ric}_{g_{J_s}} \ge \kappa g_{J_s}$ for 
all $s$. 
We also suppose additional assumptions 
which are precisely described in 
$\spadesuit$ of Section 
\ref{bdd ricci}. 
Let $g_{m,\infty}$ and $\mu_\infty$ 
be a Riemannian metric and a measure 
on $\R^n\times S^1$ defined by 
\begin{align*}
g_{m,\infty}&:=\frac{1}{m^2(1+\| y\|^2)}(dt)^2 
+\sum_{i=1}^n (dy_i)^2,\\
d\mu_\infty&:=dy_1\cdots dy_ndt,
\end{align*}
where $m$ is a positive integer, 
$y=(y_1,\ldots,y_n)\in\R^n$ 
and $e^{\sqrt{-1}t}\in S^1$. 
We define the isometric $S^1$-action 
on 
$(\R^n\times S^1,g_{m,\infty},\mu_\infty)$ 
by $(y,e^{\sqrt{-1}t})\cdot e^{\sqrt{-1}\tau}:=(y,e^{\sqrt{-1}(t+m\tau)})$ for $e^{\sqrt{-1}\tau}\in S^1$. 
The followings are the main results of this 
article. 
\begin{thm}
Let $m$ be a positive integer and 
$u\in S|_{\mu^{-1}(y)}$. 
Assume that $\mu^{-1}(y)$ is a Bohr-Sommerfeld 
fiber of $L^{\otimes m}$ and 
not a Bohr-Sommerfeld 
fiber of $L^{\otimes m'}$ for any $0<m'<m$. 
Then for some positive constant 
$K>0$, 
the family of pointed metric measure spaces with 
the isometric $S^1$-action 
\begin{align*}
\left\{ \left( S, \hat{g}_s, 
\frac{\mu_{\hat{g}_s}}{K\sqrt{s}^n},u
\right)\right\}_s
\end{align*}
converges to $\left( \R^n\times S^1, g_{m,\infty}, \mu_{\infty},(0,1)\right)$ as $s\to 0$ 
in the sense of the pointed $S^1$-equivariant 
measured Gromov-Hausdorff topology. 
\label{main 1}
\end{thm}

\begin{thm}
Let $u\in S|_{\mu^{-1}(y)}$
and assume that 
$\mu^{-1}(y)$ is not a Bohr-Sommerfeld 
fiber of $L^{\otimes m}$ for any 
positive integer $m$. 
Then $\{ (S, \hat{g}_s, 
\frac{\mu_{\hat{g}_s}}{K\sqrt{s}^n},u)\}_s$ 
converges to $(\R^n,{}^tdy\cdot dy, dy_1\cdots dy_n, 0)$ as $s\to 0$ in the sense of the pointed 
$S^1$-equivariant measured 
Gromov-Hausdorff topology. 
Here, the 
$S^1$-action on $\R^n$ is trivial. 
\label{main 2}
\end{thm}

Now let $S_\infty$ be the 
metric measure space 
appears as the limit in Theorem \ref{main 1} or 
Theorem \ref{main 2} 
and denote by $\Delta_\infty$ 
its Laplacian. 
Denote by 
$W(n+1)$ the eigenspace of 
\begin{align*}
\Delta_\infty\colon (C^\infty(S_\infty)\otimes \C)^\rho
\to (C^\infty(S_\infty)\otimes \C)^\rho
\end{align*}
associate with the eigenvalue $n+1$.
\begin{thm}
If $S_\infty$ be the 
metric measure space 
appears as the limit in Theorem \ref{main 1}, 
then $\dim W(n+1)= 1$ if $m=1$ 
and $\dim W(n+1)= 0$ if $m>1$.  
If $S_\infty$ be the 
metric measure space 
appears as the limit in Theorem \ref{main 2}, 
then $\dim W(n+1)= 0$. 
\label{main 3}
\end{thm}

This article is organized as follows. 
First of all, we explain 
how to identify the holomorphic sections of 
$L$ on $(X,J)$ with the eigenfunctions 
on the frame bundle $S$ equipped with 
the connection metric in Section 
\ref{sec line bundle} and \ref{principal metric}. 
In Section \ref{sec BS}, 
we review the definition of 
Bohr-Sommerfeld fibers 
for the pairs of symplectic manifolds and 
prequantum line bundles. 
In Section \ref{Polarizations}, 
we review the notion of Polarizations, 
which enables us to treat 
the $\omega$-compatible 
complex structures and the Lagrangian fibrations. 
In Section \ref{topology} we 
explain the notion of the pointed 
$S^1$-equivariant 
measured Gromov-Hausdorff convergence. 
This notion is the special case of the convergence 
introduced by Fukaya and Yamaguchi 
\cite{FukayaYamaguchi1994}.
These sections are the preparations for 
the main argument.
In Section \ref{Convergence}, 
we show the pointed $S^1$-equivariant 
measured 
Gromov-Hausdorff convergence 
near the Bohr-Sommerfeld fibers. 
First of all we obtain the 
local description of the connection 
metric $\hat{g}_s$ on $S$, then discuss the 
condition equivalent to the 
existence of the lower bound of the
Ricci curvatures. 
Then we show the convergence of 
$\hat{g}_s$ to $g_{m,\infty}$ as $s\to 0$. 
In Section \ref{sec others} we consider 
the limit of $\hat{g}_s$ near the non 
Bohr-Sommerfeld fibers, then show that 
the $S^1$-action on the limit space is 
trivial. 
In Section \ref{sec spectral}, 
we study the spectral structure of  
the Laplacian of the metric measure spaces 
appearing as the limit of $\hat{g}_s$.

\section{Holomorphic line bundles}\label{sec line bundle}
Let $(X,J,\omega)$ be 
a compact K\"ahler manifold. 
We write $X=X_J$ when we regard $X$ as 
a complex manifold. 
Let $\pi_E\colon E\to X_J$ be a holomorphic line bundle 
over $X_J$. 
Suppose $h$ is a hermitian metric on $E$ and 
$\nabla\colon \Gamma(E)\to\Omega^1(E)$ is the Chern connection. 
Under the decomposition $\Omega^1=\Omega^{1,0}\oplus \Omega^{0,1}$, 
we have the decomposition 
$\nabla=\nabla^{1,0}+\nabla^{0,1}$. 
Let $\nabla^*,(\nabla^{1,0})^*,(\nabla^{0,1})^*$ are the formal adjoint 
of $\nabla,\nabla^{1,0},\nabla^{0,1}$, respectively.

For a holomorphic coordinate $(U,z^1,\ldots,z^n)$ on $X_J$, 
put $\omega=\sqrt{-1}g_{i\bar{j}}dz^i\wedge d\bar{z}^j$. 
Then we may write 
\begin{align*}
\nabla^* &= (\nabla^{1,0})^* + (\nabla^{0,1})^*,\\
(\nabla^{1,0})^* &= -g^{i\bar{j}}\iota_{\partial_i}\nabla_{\bar{\partial}_j},\\
(\nabla^{0,1})^* &= -g^{i\bar{j}}\iota_{\bar{\partial}_j}\nabla_{\partial_i},
\end{align*}
where $\partial_i:=\frac{\partial}{\partial z^i}$. 
Let $F\in\Omega^{1,1}(X_J)$ 
be the curvature form. 
Since we have 
\begin{align*}
(\nabla^{1,0})^*\nabla^{1,0}s
&= (\nabla^{0,1})^*\nabla^{0,1}s
+ g^{i\bar{j}} F(\partial_i, \bar{\partial}_j)s,
\end{align*}
we obtain 
\begin{align*}
\nabla^*\nabla &= 2\Delta_{\bar{\partial}} + \Lambda_\omega F,\\
\Delta_{\bar{\partial}} &:= (\nabla^{0,1})^*\nabla^{0,1},\\
\Lambda_\omega F &:= g^{i\bar{j}} F(\partial_i, \bar{\partial}_j)\in C^\infty(X). 
\end{align*}

Let $L\to X_J$ be a holomorphic 
line bundle with hermitian metric $h$ and 
hermitian connection $\nabla$ such that the curvature 
form is equal to $-\sqrt{-1}\omega$, 
and put $E=L^k$. 
Then for the connection on $E$ determined by $\nabla$ 
we have $F = -k\sqrt{-1}\omega$, 
then 
\begin{align*}
\Lambda_\omega F 
= nk. 
\end{align*}

Now, put 
\begin{align*}
H^0(X_J,L^k):=\left\{ s\in C^\infty(L^k);\, 
\nabla^{0,1} s= 0\right\}.
\end{align*}
Since $X$ is compact, we can see 
\begin{align*}
H^0(X_J,L^k)&= \left\{ s\in C^\infty(L^k);\, 
\nabla^*\nabla s= nk s\right\}.
\end{align*}

\section{Holomorphic sections on line bundles and eigenfunctions on 
frame bundle}\label{principal metric}
Let $(X,\omega)$ be a connected 
symplectic manifold of dimension $2n$ and 
$(\pi\colon L\to X,\nabla,h)$ be a prequantum line bundle 
over $(X,\omega)$, that is, a complex line bundle with a hermitian metric $h$ 
a connection $\nabla$ preserving $h$ whose curvature form is equal to $-\sqrt{-1}\omega$.

The complex structure $J$ on $X$ is 
{\it $\omega$-compatible }if 
$\omega(J\cdot,J\cdot)=\omega$ holds and 
$g_J:=\omega(\cdot,J\cdot)$ is positive definite. 
If $J$ is $\omega$-compatible,  
then $\omega$ is a K\"ahler form on $X_J$. 

Since $\omega$ is of type $(1,1)$, 
$\nabla$ determines a holomorphic structure 
on $L$, 
consequently 
$\nabla$ is the Chern connection determined by 
$h$ and $J$.

By the previous section we have 
$\nabla^*\nabla=2\Delta_{\bar{\partial}} + n$. 
Put 
\begin{align*}
S:=S(L,h):=\{ u\in L;\, |u|_h=1\}, 
\end{align*}
which is a principal $S^1$-bundle over $X$ 
equipped with the $S^1$-connection 
$\sqrt{-1}\Gamma\in\Omega^1(S,\sqrt{-1}\R)$ 
corresponding to $\nabla$. 
The $S^1$-connection induces the following decomposition 
\begin{align*}
T_uS &:= H_u\oplus V_u,\\
H_u &:= {\rm Ker}\,(\Gamma_u\colon T_uS \to \R),\\
V_u &:= \{ \xi^\sharp_u\in T_uS;\, \xi\in \sqrt{-1}\R\},
\end{align*}
where $\xi^\sharp_u:=\frac{d}{dt}e^{t\xi}|_{t=0}$. 
Then the connection metric $\hat{g}=\hat{g}(L,J,h,\sigma,\nabla)$ on $S$ 
is defined by 
\begin{align*}
\hat{g}(L,J,h,\sigma,\nabla)
:= \sigma \cdot \Gamma^2 
+ (d\pi|_H)^* g_J
\end{align*}
for $\sigma>0$. 

\begin{rem}
\normalfont
By regarding $-\Gamma$ as a contact structure and 
$-{\sqrt{-1}}^\sharp$ as the Reeb vector field, 
$(S,\hat{g}(L,J,h,2,\nabla))$ becomes a Sasakian manifold. 
\end{rem}

Now we can recover $L$ by $S$ as the associate bundle as follows. 
Let $\rho_k\colon S^1\to GL_1(\C)$ be defined by 
$\rho_k(\lambda)=\lambda^k$ for $k\in\Z$, then 
we have the identification $L^k \cong S\times_{\rho_k}\C$. 
Then there are natural isomorphisms 
\begin{align*}
C^\infty(X,L^k) \cong (C^\infty(S)\otimes \C)^{\rho_k},
\end{align*}
where the action of $S^1$ on $C^\infty(S)\otimes \C$ is defined by 
$(\lambda\cdot f)(u):=\lambda^{k}f(u\lambda)$. 
By applying the argument in the previous section for $E=L^k$ we 
have 
$\nabla^*\nabla=2\Delta_{\bar{\partial}}+kn$. 
Note that we may regard $\nabla^*\nabla$ and $\Delta_{\bar{\partial}}$ 
as operators acting on $(C^\infty(S)\otimes \C)^{\rho_k}$, 
then by \cite[Section $3$]{Kasue2011} 
we have 
$\nabla^*\nabla=\Delta_{\hat{g}} - \frac{k^2}{\sigma}$, therefore we obtain 
\begin{align*}
2\Delta_{\bar{\partial}} = \Delta_{\hat{g}} - \left(\frac{k^2}{\sigma} + kn\right)
\colon (C^\infty(S)\otimes \C)^{\rho_k}\to (C^\infty(S)\otimes \C)^{\rho_k}.
\end{align*}

On some open set $U\subset X$, 
suppose that $L|_U$ is trivial as 
$C^\infty$ complex bundles, 
then there exists a global smooth section 
$E\in C^\infty(U,L)$ such that $h(E,E)\equiv 1$. 
Let $\gamma\in \Omega^1(U,\R)$ be defined 
by $\nabla E=\sqrt{-1}\gamma\otimes E$. 
Under the diffeomorphism 
$U\times S^1 \to S(L|_U,h)$ defined by 
$(z,e^{\sqrt{-1}t})\mapsto e^{\sqrt{-1}t}E_{z}$, 
one can obtain the following identification 
as Riemannian manifolds with isometric 
$S^1$-action; 
\begin{align}
(S|_U,\hat{g})\cong
(U\times S^1, g_J|_U + \sigma(dt+\gamma)^2 ).
\label{identification 1}
\end{align}

\section{Bohr-Sommerfeld fibers}\label{sec BS}
Let 
$(\pi\colon L\to X,\nabla)$ be a prequantum line bundle 
over a symplectic manifold $(X,\omega)$. 
A Lagrangian fibration over $(X,\omega)$ is a smooth map 
$\mu\colon X\to B$, where $B$ is a smooth manifold of dimension 
$\frac{{\rm dim}X}{2}$, such that $X_b:=\mu^{-1}(b)$ is a Lagrangian submanifold 
for every $b\in B\setminus B_{\rm sing}$ and $B\setminus B_{\rm sing}$ 
is open dense in $B$. 
We suppose that $B$ and all of the fibers $X_b$ are path-connected. 
Then every $X_b$ is diffeomorphic to a compact torus by 
Liouville-Arnold theorem. 

For a subset $Y\subset X$, 
the holonomy ${\rm Hol}(L|_Y,\nabla)$ is defined by 
\begin{align*}
{\rm Hol}(L|_Y,\nabla) := \{ e^{\sqrt{-1}t}\in S^1;\, \tilde{c}(1)=\tilde{c}(0)e^{\sqrt{-1}t},\, 
c\in \mathcal{P}(a)\}, 
\end{align*}
where $\mathcal{P}(a)$ consists of piecewise smooth curve 
$c\colon [0,1]\to X$ with $c(0)=c(1)=a\in Y$, ${\rm Im}(c)\subset Y$ and 
$\tilde{c}$ is the horizontal lift of $c$. 
Note that ${\rm Hol}(L|_Y,\nabla)$ does not depend on $a\in Y$ 
if $Y$ is path-connected.

\begin{definition}
\normalfont
\mbox{}
\begin{itemize}
\setlength{\parskip}{0cm}
\setlength{\itemsep}{0cm}
 \item[$({\rm i})$] $X_b$ is a {\it Bohr-Sommerfeld fiber of $\mu\colon X\to B$} if 
${\rm Hol}(L|_{X_b},\nabla)$ is trivial. 
 \item[$({\rm ii})$] $X_b$ is an {\it $m$-BS fiber of $\mu\colon X\to B$} if 
${\rm Hol}(L|_{X_b},\nabla)$ is a subgroup of $\Z/m\Z$. 
$X_b$ is a {\it strict $m$-BS fiber of $\mu\colon X\to B$} if 
${\rm Hol}(L|_{X_b},\nabla)\cong\Z/m\Z$. 
\end{itemize}
\end{definition}

\begin{rem}
\normalfont
$X_b$ is a $m$-BS fiber of $\mu\colon X\to B$ iff 
${\rm Hol}(L^m|_{X_b},\nabla)$ is trivial.
\end{rem}

\begin{rem}
\normalfont
In this article we suppose that 
\begin{align*}
B_m:=\{ b\in B;\, X_b\mbox{ is an }m\mbox{-BS fiber}\}
\end{align*}
are discrete in $B$ for all $m>0$. 
This condition always holds if $\mu\colon X\to B$ comes from 
completely integrable system or $B_{\rm sing}=\emptyset$ by 
the results in \cite{Tyurin2002}\cite{Tyurin2007}. 
If we put 
\begin{align*}
B_m':=\{ b\in B;\, X_b\mbox{ is a strict }m\mbox{-BS fiber}\},
\end{align*}
then $B_m=\bigsqcup_{l|m}B_l'$ holds. 
\end{rem}

\section{Polarizations}\label{Polarizations}
In this section 
we review the notion of polarizations in the 
sense of \cite{Woodhouse1992} 
to treat complex structures 
and Lagrangian fibrations 
uniformly. 

Let $V_\R$ be a real vector space 
of dimension $2n$ with 
symplectic form 
$\alpha\in \bigwedge^2 V^*$ and 
put $V=V_\R\otimes \C$. 
Then $\alpha$ extends $\C$-linearly to 
a complex symplectic form on $V$. 
{\it A Lagrangian subspace} $W${\it of} $V$ 
is a complex vector subspace of $V$ 
such that $\dim_\C W=n$ and 
$\alpha(u,v)=0$ for all $u,v\in W$. 
Put 
\begin{align*}
{\rm Lag}(V,\alpha)
:=\left\{ W\subset V;\, 
W\mbox{ is a Lagrangian subspace}
\right\},
\end{align*}
which is a submanifold of 
Grassmannian ${\rm Gr}(n,V)$. 

For a symplectic manifold $(X,\omega)$, 
put 
\begin{align*}
{\rm Lag}_\omega
:=\bigsqcup_{x\in X} {\rm Lag}(T_x X\otimes \C,\omega_x).
\end{align*}
This is a fiber bundle over $X$, and a section 
$\mathcal{P}$ of ${\rm Lag}_\omega$ 
is a subbundle of $TX\otimes \C$. 
$\mathcal{P}$ is said to be {\it integrable} 
if 
\begin{align*}
[\Gamma(\mathcal{P}|_U),\Gamma(\mathcal{P}|_U)]
\subset \Gamma(\mathcal{P}|_U)
\end{align*}
holds for any open set $U\subset X$, 
and we call such $\mathcal{P}$ a
{\it polarization of} $X$. 
In this article we consider the following two types of 
polarizations. 
\vspace{0.2cm}

\paragraph{\bf K\"ahler polarizations}
Let $J$ be an $\omega$-compatible 
complex structure. 
The subbundle 
\begin{align*}
\mathcal{P}_J:=T^{1,0}_JX \subset TX\otimes \C
\end{align*}
is called a K\"ahler polarization. 
\vspace{0.2cm}

\paragraph{\bf Real polarizations}
Let $Y$ be a
smooth manifold of dimension $n$, $\mu\colon X\to Y$ 
be a smooth map such that all $b\in \mu(X)$ are 
regular values and 
$\mu^{-1}(b)$ are Lagrangian submanifolds. 
Then 
\begin{align*}
\mathcal{P}_\mu :={\rm Ker}(d\mu)\otimes \C
\subset TX\otimes \C
\end{align*}
is called a real polarization. 
\vspace{0.2cm}

Define $l\colon {\rm Lag}(V,\alpha)\to \{ 0,1,\ldots,n\}$ 
by $l(W):=\dim_\C(W\cap \overline{W})$. 
Then for any K\"ahler polarization $\mathcal{P}_J$ we have 
$l((\mathcal{P}_J)_x)=0$, and for
any real polarization $\mathcal{P}_\mu$ we have 
$l((\mathcal{P}_\mu)_x)=n$. 

Conversely, for a polarization $\mathcal{P}$ 
such that $l(\mathcal{P}_x)=0$ for all $x\in X$, there is a 
unique complex structure $J$ 
such that $\omega(J\cdot,J\cdot)=\omega$ 
and $\mathcal{P}=T^{1,0}_JX$. 
For a polarization $\mathcal{P}$ 
such that $l(\mathcal{P}_x)=n$ for all $x\in X$, 
we obtain the Lagrangian foliation. 

Next we observe the local structure 
of ${\rm Lag}(V,\alpha)$. 
For $W\in {\rm Lag}(V,\alpha)$, 
we can take a basis $\{ w_1,\ldots,w_n\}\subset W$ 
and vectors $u^1,\ldots,u^n\in V$ 
such that 
$\{ w_1,\ldots,w_n,u^1,\ldots,u^n\}$ is a basis of $V$ and 
\begin{align*}
\alpha(w_i,w_j)=\alpha(u^i,u^j)=0,
\quad \alpha(u^i,w_j)=\delta_j^i
\end{align*}
hold. 
Put $W':={\rm span}_\C\{ u^1,\ldots,u^n\}$
and take $A\in {\rm Hom}(W,W')$. 
Then the subspace 
\begin{align*}
W_A:=\left\{ w+Aw\in V;\, w\in W\right\}
\end{align*}
is Lagrangian iff the matrix $( A_{ij})$ defined by 
$Aw_i=A_{ij}u^j$ is symmetric. 
Consequently, we have the identification 
\begin{align}
T_W {\rm Lag}(V,\alpha)
= \left\{ A\in {\rm Hom}(W,W');\, A_{ij}=A_{ji}\right\}.
\label{tangent}
\end{align}
Now, we fix $W$ such that $l(W)=n$. 
Then $w_1,\ldots,w_n,u^1,\ldots,u^n$ 
can be taken to be real vectors, hence 
\begin{align*}
l(W_A) =\dim{\rm Ker}(A-\overline{A})= n-{\rm rank}(A-\overline{A})
\end{align*}
holds. 
Moreover $W_A$ comes form an 
almost complex structure which makes $\alpha$ the positive hermitian 
iff 
${\rm Im}A\in M_n(\R)$ is the positive definite symmetric matrix. 
We define 
\begin{align*}
T_W {\rm Lag}(V,\alpha)_+
:= \left\{ A\in {\rm Hom}(W,W');\, A_{ij}=A_{ji},\, {\rm Im}A>0\right\}
\end{align*}
under the identification \eqref{tangent}.
If $W_t$ is a smooth curve in 
${\rm Lag}(V,\alpha)$ such that 
$l(W_0)=n$ and 
$\frac{d}{dt}W_t|_{t=0}\in T_{W_0} {\rm Lag}(V,\alpha)_+$, 
then there is $\delta>0$ such that 
$l(W_t)=0$ and $\alpha(w,\bar{w})>0$ for 
any $w\in W_t\setminus \{ 0\}$ and 
$0<t\le \delta$. 
Conversely, 
even 
if $W_t$ satisfies $l(W_0)=n$ and 
\begin{align*}
l(W_t)=0,\quad 
\alpha(w,\bar{w})>0 \mbox{ for 
any }w\in W_t\setminus \{ 0\}
\end{align*}
for all $t>0$, 
$\frac{d}{dt}W_t|_{t=0}$ is not 
necessary to be in 
$T_{W_0} {\rm Lag}(V,\alpha)_+$ 
since the closure of positive definite 
symmetric matrices contains semi-positive 
definite symmetric matrices.

\section{Topology}\label{topology}
In this section we explain 
the notion of the 
$S^1$-equivariant 
measured Gromov-Hausdorff topology.  
The following notion 
is the special case of 
\cite[Definition 4.1]{FukayaYamaguchi1994}. 
\begin{definition}
\normalfont 

Let $G$ be a compact topological group. 
\begin{itemize}
\setlength{\parskip}{0cm}
\setlength{\itemsep}{0cm}
 \item[(1)] 
Let $(P',d')$ and $(P,d)$ be metric spaces with 
isometric $G$-action. 
A map $\phi:P'\to P$ is an {\it $G$-equivariant 
$\varepsilon$-approximation} if $\phi$ is 
$G$-equivariant and $\varepsilon$-approximation. 
Here, $\varepsilon$-approximation means that 
$|d'(x',y') - d(\phi(x'),\phi(y'))| < \varepsilon$ 
holds for all $x',y'\in P'$ and 
$P\subset B(\phi(P'),\varepsilon)$. 
Moreover if $\phi$ is a Borel map then 
it is called a {\it Borel $G$-equivariant 
$\varepsilon$-approximation}. 
 \item[(2)] Let $\{(P_i,d_i,\nu_i,p_i)\}_i$ be a sequence of pointed 
metric measure spaces with isometric $G$-action. 
$(P_\infty,d_\infty,\nu_\infty,p_\infty)$ is said to be 
{\it the pointed $G$-equivariant measured Gromov-Hausdorff limit of 
$\{(P_i,d_i,\nu_i,p_i)\}_i$} 
if $G$ acts on $P_\infty$ isometrically and 
for any $R>0$ there are positive numbers 
$\{ \varepsilon_i\}_i$, $\{ R_i\}_i$ with 
\begin{align*}
\lim_{i\to \infty}\varepsilon_i = 0,\quad \lim_{i\to \infty}R_i=R,
\end{align*}
and Borel $G$-equivariant 
$\varepsilon_i$-approximation 
\begin{align*}
\phi_i\colon (\pi_i^{-1}(B(x_i,R_i)),p_i)\to (\pi_\infty^{-1}(B(x_\infty,R)),p_\infty)
\end{align*}
for every $i$ 
such that ${\phi_i}_*(\nu_i|_{\pi_i^{-1}(B(x_i,R_i))} )
\to \nu_\infty|_{\pi_\infty^{-1}(B(x_\infty,R))}$ 
vaguely. 
Here, $\pi\colon P_i\to P_i/G$ is the quotient map and 
$x_i=\pi_i(p_i)$. 
\end{itemize}
\label{def GmGH}
\end{definition}

\section{Convergence}\label{Convergence}
Throughout of this section 
let $(X^{2n},\omega)$ be a symplectic manifold, 
$Y^n$ a smooth manifold 
and 
\begin{align*}
\mu\colon X\to Y
\end{align*}
be a smooth surjective map such that $\mu^{-1}(y)$ are smooth 
compact connected Lagrangian submanifolds 
for all regular value $y\in Y$. 
Assume that $y_0\in Y$ is 
a regular value of $\mu$. 
Then by \cite{ArnoldAvez1968french}\cite{Duistermaat1980}\cite{MarkusMeyer1974}, 
there are open neighborhoods 
$U\subset X$ of $X_0 := \mu^{-1}(y_0)$, 
$B'\subset Y$ of $y_0$, 
$B\subset \R^n$ of the origin $0$, 
diffeomorphisms $\tilde{f}\colon B\times T^n \stackrel{\cong}{\to} U$ and 
$f\colon B'\stackrel{\cong}{\to} B$ such that 
$\tilde{f}^*\omega = \sum_{i=1}^n dx_i\wedge d\theta^i,$ and $f(y_0) = 0$, 
where $x=(x_1,\ldots,x_n)=f\circ\mu\circ \tilde{f}$ 
and $\theta=(\theta^1\ldots,\theta^n)\in T^n=\R^n/\Z^n$. 
Therefore, we may suppose 
\begin{align*}
U&=B\times T^n,\quad \mu=(x_1,\ldots,x_n),\quad
\omega = dx_i\wedge d\theta^i,\\
B&=\left\{ x=(x_1,\ldots,x_n)\in\R^n;\, \| x\|
=\sqrt{x_1^2+\cdots + x_n^2}<R\right\},\\
X_0&=\{ 0\}\times T^n
\end{align*}
for some $0<R\le 1$. 

Let $(L,\nabla)$ be the prequantum line bundle on 
$(X,\omega)$ and $h$ be a hermitian metric such that $\nabla h=0$. 
Since $[\omega|_{U}] = 0\in H^2(U)$, 
then the 1st Chern class of $(L,\nabla)|_{U}$ vanishes, 
hence $L|_{U}$ is trivial as $C^\infty$ complex line bundle 
by \cite[Section $5$]{Chern1995}.

From now on we consider some 
covering spaces of $U$ given by the followings. 
Let $\Phi\colon\Z^n\to \Z/m\Z$ 
be a homomorphism of $\Z$-modules. 
Then ${\rm Ker}\,\Phi$ is of 
rank $n$, hence $\R^n/{\rm Ker}\,\Phi$ is 
diffeomorphic to the $n$-dimensional torus. 
Now we have the natural projection 
\[ 
\left.
\begin{array}{ccc}
\R^n/{\rm Ker}\,\Phi & \rightarrow & T^n \\
\rotatebox{90}{$\in$} & & \rotatebox{90}{$\in$} \\
\theta\, {\rm mod}\,{\rm Ker}\,\Phi & \mapsto & \theta\, {\rm mod}\,\Z^n
\end{array}
\right.
\]
which give a covering space and 
a covering map 
\begin{align*}
U_\Phi:=B\times \left( 
\R^n/{\rm Ker}\,\Phi\right),\quad 
p_\Phi\colon U_\Phi\to U. 
\end{align*}
From now on we denote by 
$\theta$ the element of $\R^n/{\rm Ker}\,\Phi $ 
or $T^n$ for the simplicity, 
if there is no fear of confusion. 
If we take $\mathbf{w}\in\Z^n$ 
then 
\[ 
\left.
\begin{array}{cccc}
\beta(\Phi(\mathbf{w})):
& U_\Phi & \rightarrow & U_\Phi \\
&\rotatebox{90}{$\in$} & & \rotatebox{90}{$\in$} \\
&(x,\theta) & \mapsto & (x,\theta+\mathbf{w})
\end{array}
\right.
\]
gives the action of ${\rm Im}\,\Phi$ 
on $U_\Phi$, which is 
the deck transformations of $p_\Phi$.

\begin{prop}
Let $X_0$ be a strict $m$-BS fiber. 
Then there are surjective 
homomorphism $\Phi\colon\Z^n\to\Z/m\Z$ 
and
$E\in C^\infty(p_\Phi^*L)$ such that 
$h(E,E)\equiv 1$ and 
$\nabla E = -\sqrt{-1}x_id\theta^i\otimes E$. 
Moreover, 
the deck transformations of $p_\Phi$ 
satisfies 
$\beta(k)^*E = e^{\frac{2k\sqrt{-1}\pi}{m}}E$ 
for $k\in\Z/ m\Z$. 
\label{normalize of connection}
\end{prop}
\begin{proof}
Since $X_0$ is the $m$-BS fiber, one can obtain 
the flat section $\hat{E}$ of $(L^m|_{U})|_{x=0})$ such that 
$h^{\otimes m}(\hat{E},\hat{E})\equiv 1$. 
Then $\hat{E}$ can be extended to the nowhere vanishing 
section of $C^\infty(L^m|_{U})$ with $h^{\otimes m}(\hat{E},\hat{E})\equiv 1$. 
Define $\gamma\in\Omega^1(U)$ by 
$\nabla\hat{E}=\sqrt{-1}\gamma\otimes \hat{E}$. 
By computing the curvature form of 
$\nabla$ one obtain 
$d\gamma = -m\omega|_{U} = -mdx_i\wedge d\theta^i$ which implies that 
$\gamma+mx_id\theta^i$ is a closed $1$-form 
on $U$. 
Denote by $\alpha$ the cohomology class represented by 
$\gamma+mx_id\theta^i$ and 
let $\iota\colon \{ 0\}\times T^n \to U$ be the 
natural embedding. 
Since $\hat{E}|_{x=0}$ is flat, then one can see that 
$\iota^*\gamma=0$ and $\iota^*\alpha=0$. 
Since $\iota^*\colon H^1(B\times T^n) \to H^1(\{ 0\}\times T^n)$ 
is isomorphic, one can see that $\alpha=0$, 
therefore there exists $\tau\in C^\infty(U,\R)$ such that 
$\gamma+mx_id\theta^i = d\tau$. 

Then one have 
\begin{align*}
\nabla( e^{-\sqrt{-1}\tau}\hat{E})
= \sqrt{-1}(-d\tau + \gamma) \otimes e^{-\sqrt{-1}\tau}\hat{E}
= -m\sqrt{-1}x_id\theta^i \otimes e^{-\sqrt{-1}\tau}\hat{E},
\end{align*}
accordingly, by replacing 
$e^{-\sqrt{-1}\tau}\hat{E}$ by $\hat{E}$, 
we may suppose  
\begin{align*}
\nabla \hat{E}=-m\sqrt{-1}x_id\theta^i 
\otimes \hat{E}.
\end{align*}

Let $\tilde{p}\colon 
\tilde{U}=B\times \R^n\to B\times T^n$ 
be the universal cover of $U$. 
Then there is a nowhere vanishing section 
$E\in C^\infty(\tilde{p}^*L)$ such that 
$E^{\otimes m}=\tilde{p}^*\hat{E}$, 
consequently we obtain 
the homomorphism $\Phi\colon
\pi_1(U)=\Z^n \to \Z/m\Z$ defiend by 
\begin{align*}
E_{(x,\theta+\mathbf{k})} = 
e^{2\pi\sqrt{-1}\Phi(\mathbf{k})} E_{(x,\theta)}
\end{align*}
for $\mathbf{k}\in \Z^n$. 
Therefore, $E$ descends to 
the section of $p_\Phi^*L$, 
then 
\begin{align*}
\nabla E=-\sqrt{-1}x_id\theta^i
\otimes E
\end{align*}
holds. 
Since $X_0$ is the strict $m$-BS fiber, 
$\Phi$ is surjective and $p_\Phi$ is an 
$m$-fold covering. 
\end{proof}

\subsection{Local description of the 
complex structures and the metrics}

We assume that an $\omega$-compatible 
complex structure $J$ on $X$ is given such that 
$\mathcal{P}_J|_U$ is close to 
$\mathcal{P}_\mu|_U$, 
as sections of ${\rm Lag}_\omega|_U\to U$. 
Define $\mathcal{P}_\mu'$ by 
\begin{align*}
(\mathcal{P}_\mu')_p:=
{\rm span}_\C\left\{ \left( 
\frac{\partial}{\partial x_1}\right)_p, 
\ldots,
\left( \frac{\partial}{\partial x_n}\right)_p
\right\} \subset T_pU\otimes\C, 
\end{align*}
then we have the direct decomposition 
$TU\otimes \C = \mathcal{P}_\mu \oplus \mathcal{P}_\mu'$. 
Since $\mathcal{P}_J|_U$ is close to 
$\mathcal{P}_\mu|_U$, 
the identification \eqref{tangent} gives 
\begin{align*}
A=\left( A_{ij}(x,\theta)\right)_{i,j}\in 
C^\infty(U)\otimes M_n(\C)
\end{align*}
such that 
\begin{align*}
A_{ij}=A_{ji},\quad 
{\rm Im}A>0
\end{align*}
and 
\begin{align*}
\frac{\partial}{\partial\theta^i}
+A_{ij}(x,\theta) \frac{\partial}{\partial x_j},\quad 
i=1,\ldots,n
\end{align*}
is a frame of $\mathcal{P}_J|_U$.
Moreover the integrability of $J$ gives 
\begin{align}
\frac{\partial A_{jk}}{\partial\theta^i} 
- \frac{\partial A_{ik}}{\partial\theta^j} 
+A_{il}\frac{\partial A_{jk}}{\partial x_l} 
-A_{jl}\frac{\partial A_{ik}}{\partial x_l} = 0.
\label{integrable}
\end{align}

Conversely, if a complex matrix valued function 
$A$ 
satisfies above properties then we can recover 
$J|_U$. 
Therefore, the $\omega$-compatible $J$
complex structure close to $\mathcal{P}_\mu$ 
is identified with the matrix valued function 
$A$ on $U$. 

If we put $A_{ij}=P_{ij}+\sqrt{-1}Q_{ij}$, where $P_{ij},Q_{ij}\in\R$, 
and denote by $(Q^{ij})$ the inverse of $(Q_{ij})$, 
then one can see 
\begin{align}
J\left( \frac{\partial}{\partial\theta^i}  \right)
&= -P_{ij}Q^{jk}\frac{\partial}{\partial\theta^k} 
-(Q_{ik} + P_{ij}Q^{jl}P_{lk})\frac{\partial}{\partial x_k},\label{almost cpx1}\\
J\left( \frac{\partial}{\partial x_i}  \right)
&= Q^{ik}\frac{\partial}{\partial\theta^k} 
+Q^{ij}P_{jk}\frac{\partial}{\partial x_k},\label{almost cpx2}\\
Jd\theta^k
&= -P_{ij}Q^{jk}d\theta^i +Q^{ik}dx_i,\label{almost cpx3}\\
Jdx_k
&= -(Q_{ik} + P_{ij}Q^{jl}P_{lk})d\theta^i +Q^{ij}P_{jk}dx_i,\label{almost cpx4}
\end{align}
therefore we obtain
\begin{align*}
g_J|_U &= 
g_A:= (Q_{ij} + P_{ik}Q^{kl}P_{lj}) d\theta^i d\theta^j
-2 P_{ik}Q^{jk} d\theta^i dx_j
+ Q^{ij} dx_i dx_j.
\end{align*}

Denote by $d_g$ the Riemannian distance of 
a Riemannian metric $g$. 
Then $g_J|_U = g_A$, $d_{g_J}|_U \le d_{g_A}$ 
always holds, however, 
the opposite inequality does not hold in general 
since the shortest path connecting 
two points in $U$ need not be included in $U$. 
Here we consider the 
lower estimate of $d_{g_J}$ 
and the upper estimate of $d_{g_A}$.

For a real symmetric positive definite 
matrix valued function 
$S(x,\theta)=(S_{ij}(x,\theta))_{i,j}$ depending on 
$(x,\theta)\in U$ continuously, 
let $\lambda_1(x,\theta),\ldots,\lambda_n(x,\theta)$ 
be the eigenvalues of $S(x,\theta)$. 
Define 
\begin{align*}
U_r &:= \{ (x,\theta) \in\R^n\times T^n;\, \| x\|<r\}
\subset U\quad (r\le R),\\
\sup S &:= \sup_{i,(x,\theta)
\in U_{\frac{R}{2}}}
\lambda_i(x,\theta),
\quad
\inf S := \inf_{i,(x,\theta)
\in U_{\frac{R}{2}}}
\lambda_i(x,\theta).
\end{align*}
Since $\overline{U_{\frac{R}{2}}}$ is compact, 
$0<\inf S\le\sup S<\infty$ holds.

\begin{prop}
Put 
\begin{align*}
\Theta:=Q+PQ^{-1}P
\end{align*}
for $A=P+\sqrt{-1}Q$. 
The following inequalities 
\begin{align*}
\sqrt{ \inf (\Theta^{-1}) }\,\| x-x'\|
&\le d_{g_J}(u,u'),\\
d_{g_A}(u,u')
&\le \sqrt{\sup (\Theta^{-1})}\| x-x'\| 
+ \frac{\sqrt{n\sup \Theta}}{2}
\end{align*}
hold for any $u=(x,\theta), 
u'=(x',\theta')\in U_{\frac{R}{2}}$. 
\label{estimate on distance}
\end{prop}
\begin{proof}
First of all we show the first equality. 
If we write 
\[ d\theta=
\left (
\begin{array}{c}
d\theta^1 \\
\vdots \\
d\theta^n
\end{array}
\right ),\quad 
dx=
\left (
\begin{array}{c}
dx_1 \\
\vdots \\
dx_n
\end{array}
\right ),\quad 
x=
\left (
\begin{array}{c}
x_1 \\
\vdots \\
x_n
\end{array}
\right ),
\]
then we may write 
\begin{align*}
g_A
&=
{}^td\theta \cdot \Theta \cdot d\theta
- {}^t dx\cdot Q^{-1}P\cdot d\theta
- {}^t d\theta\cdot PQ^{-1} \cdot dx
+ {}^t dx\cdot Q^{-1}\cdot dx\\
&=
{}^t\left( \sqrt{\Theta}d\theta - \sqrt{\Theta^{-1}}
PQ^{-1} dx\right)\cdot 
\left( \sqrt{\Theta}d\theta - \sqrt{\Theta^{-1}}
PQ^{-1} dx\right)\\
&\quad\quad
+ {}^t dx\cdot 
\left( Q^{-1}-Q^{-1}P\Theta^{-1}P Q^{-1}\right)\cdot dx.
\end{align*}
Since we have 
\begin{align*}
&\quad\ \Theta\left( Q^{-1}
-Q^{-1}P\Theta^{-1}P Q^{-1}\right)\\
&=1+PQ^{-1}PQ^{-1}
-P\Theta^{-1}P Q^{-1}
-PQ^{-1}PQ^{-1}P\Theta^{-1}P Q^{-1}\\
&=1+PQ^{-1}PQ^{-1}
-PQ^{-1}\left( Q+ PQ^{-1}P\right) 
\Theta^{-1}P Q^{-1}\\
&=1+PQ^{-1}PQ^{-1}
-PQ^{-1}P Q^{-1}=1,
\end{align*}
we can see 
\begin{align*}
\Theta^{-1}=Q^{-1}
-Q^{-1}P\Theta^{-1}P Q^{-1}.
\end{align*}
Therefore, 
\begin{align}
g_A
&=
{}^t\left( \sqrt{\Theta}d\theta - \sqrt{\Theta^{-1}}
PQ^{-1} dx\right)\cdot 
\left( \sqrt{\Theta}d\theta - \sqrt{\Theta^{-1}}
PQ^{-1} dx\right)\notag\\
&\quad\quad
+ {}^t dx\cdot 
\Theta^{-1}\cdot dx
\label{principal metric eq}
\end{align}
holds. 
Now let $c_1\colon [0,1]\to X$ be a 
path connecting 
$u,u'
\in U_{\frac{R}{2}}$, and 
put $u=(x,\theta)$ 
and $u'=(x',\theta')$ 
with $\| x\|,\| x'\|<\frac{R}{2}$. 
Note that the image of $c_3$ is not  
always contained in 
$U_{\frac{R}{2}}$. 
If ${\rm Im}(c_1) \subset U_{\frac{R}{2}}$ 
does not hold, 
then let 
\begin{align*}
\tau_0:=\inf\{ \tau\in [0,1];\, c_3(\tau)
\notin U_{\frac{R}{2}}\}. 
\end{align*}
If ${\rm Im}(c_1) \subset U_{\frac{R}{2}}$ 
holds, then put 
$\tau_0:=1$. 
Put $c_1(\tau)=(x(\tau),\theta(\tau))$. 
Then by \eqref{principal metric eq} 
we can see 
\begin{align*}
\mathcal{L}(c_1)
&\ge \int_0^{\tau_0}\sqrt{ {}^t x'(\tau)\cdot
\Theta^{-1}\cdot x'(\tau)}d\tau\\
&\ge \sqrt{ \inf (\Theta^{-1}) }
\int_0^{\tau_0}|x'(\tau)|d\tau
\ge \sqrt{ \inf (\Theta^{-1}) }\,\| x-x'\|. 
\end{align*}

Next we show the second inequality. 
To show it, we compute the length of 
two types of paths in 
$U_{\frac{R}{2}}$.

For $\theta\in \R^n$ put 
$c_2(\tau):=(x,\tau \theta)$, then \eqref{principal metric eq} 
gives
\begin{align*}
\mathcal{L}(c_2)=\int_0^1|c_1'(\tau)|_{g_A}d\tau
&=\int_0^1\sqrt{ \Theta_{ij} \theta^i\theta^j}d\tau\\
&\le \sqrt{ \sup\Theta}\| \theta\|.
\end{align*}
If $c_3(\tau):=(\tau x+(1-\tau)x',\theta)$, 
where $\| x\|\le \frac{R}{2}$, then 
\begin{align*}
\mathcal{L}(c_3)=\int_0^1|c_3'(\tau)|_{\hat{g}_A}d\tau
&=\int_0^1
\sqrt{\Theta^{ij} (x_i-x'_i)(x_j-x'_j)}d\tau\\
&\le \sqrt{\sup (\Theta^{-1})}\| x-x'\|.
\end{align*}
Connecting these two types of paths 
one can see
\begin{align*}
d_A(u,u')
&\le 
\sqrt{\sup (\Theta^{-1})}\| x\|
+\sqrt{\sup\Theta}\cdot
{\rm diam}(T^n)\\
&= \sqrt{\sup (\Theta^{-1})}\| x\|
+\frac{\sqrt{n\sup\Theta}}{2}.
\end{align*}
\end{proof}

Now, we describe Riemannian metric 
$\hat{g}(L|_U,J,h,\sigma,\nabla)$ 
using the identification \eqref{identification 1} 
in the case of $X_0$ is a strict $m$-BS fiber. 
First of all we consider the connection metric 
with respect to the pullback 
of $g_J$ and $L|_U$ 
by the covering map $p_\Phi\colon U_\Phi\to U$, 
which is obtained in Proposition \ref{normalize of connection}. 
We also denote by $p_\Phi\colon {p_\Phi}^* L\to 
L|_U$ 
the lift of the covering map, then 
the following commutative diagram is obtained; 
\[ 
\left.
\begin{array}{ccc}
p_\Phi^* L & \rightarrow & L|_U \\
\downarrow & \circlearrowleft & \downarrow \\
 U_\Phi & \to & U
\end{array}
\right.
\]
Let $p_\Phi^*J$ be the complex structure on 
$U_\Phi$ inherited from $U$ by the covering map. 
Then one can see 
\begin{align*}
S({p_\Phi}^* L,{p_\Phi}^* h) = p_\Phi^{-1}(S(L,h))
\end{align*}
and 
\begin{align*}
\hat{g}({p_\Phi}^* L, {p_\Phi}^* J, {p_\Phi}^* h, \sigma, {p_\Phi}^*\nabla)
={p_\Phi}^* \hat{g}(L|_U,J,h,\sigma,\nabla).
\end{align*}

Since ${p_\Phi}^* L$ is trivial as $C^\infty$ 
complex line bundle, 
there is the identification 
\[ 
\left.
\begin{array}{ccc}
U_\Phi\times S^1 & \rightarrow & 
S({p_\Phi}^* L,{p_\Phi}^* h) \\
\rotatebox{90}{$\in$} & & \rotatebox{90}{$\in$} \\
(x,\theta,e^{\sqrt{-1}t}) & \mapsto & e^{\sqrt{-1}t}\cdot E_{(x,\theta)}
\end{array}
\right.
\]
by \eqref{identification 1}, 
where $E\in C^\infty({p_\Phi}^*L)$ is taken 
as in Proposition 
\ref{normalize of connection}. 
Under the identification we have 
\begin{align*}
\hat{g}({p_\Phi}^* L, {p_\Phi}^* J, {p_\Phi}^* h, \sigma, {p_\Phi}^*\nabla)
&= \sigma(dt - x_id\theta^i)^2\\
&\quad\quad + 
(Q_{ij} + P_{ik}Q^{kl}P_{lj}) d\theta^i d\theta^j\\
&\quad\quad - 2 P_{ik}Q^{jk} d\theta^i dx_j
+ Q^{ij} dx_i dx_j.
\end{align*}
By Proposition \ref{normalize of connection}, 
the deck transformation of 
\begin{align*}
p_\Phi\colon (S({p_\Phi}^*L,{p_\Phi}^*h)) \to S(L|_U,h)
\end{align*}
is identified with 
\begin{align}
k\cdot(x,\theta,e^{\sqrt{-1}t})
:=(x,\theta+k\mathbf{w}_0, e^{\sqrt{-1}(t-\frac{2k\pi}{m})})
\quad(k\in\Z/m\Z), 
\label{deck}
\end{align}
where $\mathbf{w}_0\in\Z^n$ is 
taken such that $\Phi(\mathbf{w}_0)
=1\in\Z/m\Z$. 
Thus we obtain the next proposition.

\begin{prop}
Define the Riemannian metric $\hat{g}_A$ 
on $U_\Phi\times S^1$ by 
\begin{align*}
\hat{g}_A
&= \sigma(dt - x_id\theta^i)^2
+ (Q_{ij} + P_{ik}Q^{kl}P_{lj}) d\theta^i d\theta^j\\
&\quad\quad - 2 P_{ik}Q^{jk} d\theta^i dx_j
+ Q^{ij} dx_i dx_j,
\end{align*}
which is invariant under the 
$\Z/m\Z$ action defined by \eqref{deck}. 
If $X_0$ 
is a strict $m$-BS fiber, then 
\begin{align*}
{p_\Phi}^*\, \hat{g}(L|_U,J|_U,h,\sigma,\nabla) 
= \hat{g}_A
\end{align*}
holds. 
\label{description of connection metric}
\end{prop}

\subsection{Boundedness of the Ricci curvatures}\label{bdd ricci}
First of all we compute the Ricci 
curvature of $g_J|_U$. 
Since $\omega$ is the K\"ahler form on 
$(U,J)$, it suffices to compute the Ricci form 
of $\omega$. 
First of all we can see that 
\begin{align*}
\del\theta^i\left( 
\frac{\del}{\del\theta^j}+A_{jk}\frac{\del}{\del x_k}\right)
= d\theta^i\left( 
\frac{\del}{\del\theta^j}+A_{jk}\frac{\del}{\del x_k}\right) =\delta^i_j,
\end{align*}
hence $\del\theta^1,\ldots,\del\theta^n$ 
forms the dual frame of $\Omega^{1,0}$.

\begin{prop}
The K\"ahler form $\omega|_U$ and 
the Ricci form $\rho_\omega|_U$ are 
given by 
\begin{align*}
\omega|_U&=2\sqrt{-1}Q_{ij}\del\theta^i\wedge 
\delb\theta^j,\\
\rho_\omega|_U&=\sqrt{-1}\del\delb\log \det (Q_{ij}) 
- \sqrt{-1}\del\alpha + \sqrt{-1}\,
\overline{\del\alpha},
\end{align*}
where 
\begin{align*}
\alpha &:= \frac{\del \bar{A}_{ij}}{\del x_i}\delb\theta^j\in 
\Omega^{0,1}(U). 
\end{align*}
\label{Ricci form}
\end{prop}
\begin{proof}
Since $dx_i-A_{ij}d\theta^j$ is of type 
$(0,1)$, one can see 
$\del x_i=A_{ij}\del\theta^j$. Then 
we have 
\begin{align*}
\omega|_U=dx_i\wedge d\theta^i
=\del x_i\wedge \delb \theta^i
+ \delb x_i\wedge \del \theta^i
= 2\sqrt{-1}Q_{ij}\del \theta^i
\wedge \delb \theta^j.
\end{align*}

Take $f\in C^\infty(U',\C^\times)$ such that 
$\Omega:=f\del\theta^1\wedge\cdots\wedge\del\theta^n$ 
is a nowhere vanishing holomorphic section 
of the canonical bundle $K_X|_{U'}$ on some 
open set $U'\subset U$. 
If we put $\beta=f^{-1}\delb f$, 
then the Ricci form $\rho_\omega|_{U'}$ 
is given by 
\begin{align*}
-\sqrt{-1}\del\delb\log \frac{\omega|_{U'}^n}{\Omega\wedge \overline{\Omega}}
&= -\sqrt{-1}\del\delb\log \det (Q_{ij}) 
+ \sqrt{-1}\del\delb\log |f|^2\\
&= -\sqrt{-1}\del\delb\log \det (Q_{ij}) 
+ \sqrt{-1}\del\beta - \sqrt{-1}\,\overline{\del\beta}.
\end{align*}
Since we have 
\begin{align}
0 = f^{-1}\delb \Omega
= \beta\wedge \del\theta^1\wedge\cdots\wedge\del\theta^n
+ \delb (\del\theta^1\wedge
\cdots\wedge\del\theta^n),\label{delb of hol vol}
\end{align}
it suffices to compute 
$\delb \del\theta^i$ to describe $\beta$. 
Now, we have
\begin{align*}
&\quad \delb \del\theta^i
\left( \frac{\del}{\del\theta^k}+A_{kj}\frac{\del}{\del x_j}, 
\frac{\del}{\del\theta^l}+\bar{A}_{lh}\frac{\del}{\del x_h}
\right)\\
&= d\del\theta^i
\left( \frac{\del}{\del\theta^k}+A_{kj}\frac{\del}{\del x_j}, 
\frac{\del}{\del\theta^l}+\bar{A}_{lh}\frac{\del}{\del x_h}
\right)\\
&= -\del\theta^i
\left( \left[ \frac{\del}{\del\theta^k}+A_{kj}\frac{\del}{\del x_j}, 
\frac{\del}{\del\theta^l}+\bar{A}_{lh}\frac{\del}{\del x_h}
\right]\right)\\
&= -\left(\frac{\del \bar{A}_{lh}}{\del\theta^k}+A_{kj}\frac{\del \bar{A}_{lh}}{\del x_j}\right)\del\theta^i
\left(\frac{\del}{\del x_h}\right)
+
\left(\frac{\del A_{kj}}{\del\theta^l}+\bar{A}_{lh}\frac{\del A_{kj}}{\del x_h}\right)\del\theta^i
\left(\frac{\del}{\del x_j}\right).
\end{align*}
Since 
\begin{align*}
\frac{\del}{\del x_h}
= \frac{Q^{hl}}{2\sqrt{-1}}
\left( \frac{\del}{\del\theta^l}+A_{lk}\frac{\del}{\del x_k}
- \frac{\del}{\del\theta^l}
-\bar{A}_{lk}\frac{\del}{\del x_k}\right)
\end{align*}
holds, we have 
$\del\theta^i
\left(\frac{\del}{\del x_h}\right)
=\frac{Q^{hi}}{2\sqrt{-1}}$, 
which gives 
\begin{align*}
\delb \del\theta^i=
- \frac{Q^{hi}}{2\sqrt{-1}}\left(
\frac{\del \bar{A}_{lh}}{\del\theta^k}+A_{kj}\frac{\del \bar{A}_{lh}}{\del x_j}
- \frac{\del A_{kh}}{\del\theta^l}
- \bar{A}_{lj}\frac{\del A_{kh}}{\del x_j}
\right)\del\theta^k \wedge 
\delb\theta^l.
\end{align*}
Moreover, the integrability of $J$ implies 
\begin{align*}
\frac{\del \bar{A}_{lh}}{\del\theta^k}+A_{kj}\frac{\del \bar{A}_{lh}}{\del x_j}
&= \frac{\del \bar{A}_{lh}}{\del\theta^k}+\bar{A}_{kj}\frac{\del \bar{A}_{lh}}{\del x_j}
+2\sqrt{-1}Q_{kj}\frac{\del \bar{A}_{lh}}{\del x_j}\\
&= \frac{\del \bar{A}_{kh}}{\del\theta^l}+\bar{A}_{lj}\frac{\del \bar{A}_{kh}}{\del x_j}
+2\sqrt{-1}Q_{kj}\frac{\del \bar{A}_{lh}}{\del x_j},
\end{align*}
accordingly one can see that 
\begin{align}
\delb \del\theta^i
&=
Q^{hi}\left( 
\frac{\del Q_{kh}}{\del\theta^l}
+ \bar{A}_{lj}\frac{\del Q_{kh}}{\del x_j}
- Q_{kj}\frac{\del \bar{A}_{lh}}{\del x_j}
\right)\del\theta^k \wedge 
\delb\theta^l.\label{ddbar}
\end{align}

By combining \eqref{delb of hol vol}, 
we have 
\begin{align*}
\beta
&= \left( Q^{ih}
\frac{\del Q_{ih}}{\del\theta^l}
+ \bar{A}_{lj}Q^{ih}\frac{\del Q_{ih}}{\del x_j}
- \frac{\del \bar{A}_{lj}}{\del x_j}
\right)\delb\theta^l. 
\end{align*}
Now the Jacobi's formula yields 
\begin{align*}
\delb (\log \det(Q_{ij}))
= Q^{ih}\delb Q_{ih}
&= Q^{ih}\left( \frac{\del Q_{ih}}{\del \theta^l}
\delb\theta^l
+ \frac{\del Q_{ih}}{\del x^j}
\delb x^j \right)\\
&= Q^{ih}\left( \frac{\del Q_{ih}}{\del \theta^l}
+ \bar{A}_{jl}\frac{\del Q_{ih}}{\del x^j}
\right)\delb\theta^l,
\end{align*}
therefore, we obtain 
\begin{align*}
\beta
= \delb (\log \det(Q_{ij}))
- \frac{\del \bar{A}_{lj}}{\del x_j}\delb\theta^l,
\end{align*}
which gives the assertion. 
\end{proof}

\begin{prop}
Let $\alpha$ be as in Proposition 
$\ref{Ricci form}$. 
Then we have 
\begin{align*}
\del\alpha
&= \left( \frac{\del^2 \bar{A}_{il}}{\del\theta^k\del x_i} 
+ A_{kh}\frac{\del^2 \bar{A}_{il}}{\del x_h\del x_i}\right)\del \theta^k
\wedge \delb\theta^l
\\
&\quad \quad 
- Q^{mh}\frac{\del \bar{A}_{im}}{\del x_i}
\left( \frac{\del Q_{kh}}{\del\theta^l}
+ \bar{A}_{lj}\frac{\del Q_{kh}}{\del x_j}
- Q_{kj} \frac{\del \bar{A}_{lh}}{\del x_j}
\right)
\del\theta^k\wedge\delb\theta^l.
\end{align*}
\label{Ricci form 2}
\end{prop}
\begin{proof}
Since 
\begin{align*}
\del\alpha
&=\del \left( \frac{\del \bar{A}_{il}}{\del x_i}\delb\theta^l\right)\\
&=\left( \frac{\del^2 \bar{A}_{il}}{\del\theta^k\del x_i} 
+ A_{kh}\frac{\del^2 \bar{A}_{il}}{\del x_h\del x_i}\right)\del\theta^k\wedge \delb\theta^l
+\frac{\del \bar{A}_{il}}{\del x_i} \del \delb\theta^l,
\end{align*}
the assertion follows from \eqref{ddbar}.
\end{proof}

From now on we consider the one parameter 
family of $\omega$-compatible complex structures 
$\{ J_s\}_{0<s<\delta}$ on 
$(X,\omega)$. 
Then we denote by $A(s,\cdot)$ the matrix 
valued function 
corresponding to $J_s|_U$. 
For simplicity, we often write 
$A=A(s,\cdot)$ if there is no fear 
of confusion. 
We assume the following condition $\spadesuit$ 
for $\{ J_s\}$. 
Let ${\rm pr}\colon X\times [0,\delta)\to X$ 
be the projection and 
${\rm pr}^*{\rm Lag}_\omega$ be the 
pullback bundle. 
\begin{itemize}
\setlength{\parskip}{0cm}
\setlength{\itemsep}{0cm}
 \item[$\spadesuit$] 
There is a smooth section 
$\mathcal{P}$ of ${\rm pr}^*{\rm Lag}_\omega|_{U\times [0,\delta)}
\to U\times [0,\delta)$ such that 
$\mathcal{P}(\cdot,s)=\mathcal{P}_{J_s}|_U$ for 
$s>0$, $\mathcal{P}(\cdot,0)
=\mathcal{P}_\mu|_U$ and 
\begin{align*}
\frac{d}{ds}\mathcal{P}(x,s)\Big|_{s=0}\in 
T_{\mathcal{P}_\mu(x)}{\rm Lag}(T_xX\otimes \C, 
\omega_x)_+.
\end{align*}
\end{itemize}
By assuming 
$\spadesuit$, 
there are a constant $K>0$ and 
$A^0\in C^\infty(U)\otimes M_n(\C)$ 
such that 
$\sup_{i,j}\| A_{ij}(s,\cdot)-sA^0_{ij}
\|_{C^2(U)} \le Ks^2$, ${\rm Im}(A^0)$ is a 
positive definite symmetric matrix and 
$\sup_{i,j}\| A^0_{ij}\|_{C^2(U)}<\infty$.

For a function $f_0(s,x,\theta)$ and 
$f_1(s,x,\theta)$ we write 
\begin{align*}
f_0(s,x,\theta) = f_1(s,x,\theta)
+\mathcal{O}_{C^l}(s^k)
\end{align*}
if there exists a constant $K>0$ 
such that $\| f_0(s,x,\theta) 
- f_1(s,x,\theta)\|_{C^l(U)}\le Ks^k$. 
For instance, if $\{ J_s\}_{s}$ 
satisfies $\spadesuit$, 
then we may write 
\begin{align*}
A_{ij} = sA^0_{ij}+\mathcal{O}_{C^2}(s^2).
\end{align*}

\begin{prop}
Assume that $\{ J_s\}_{s}$ 
satisfies $\spadesuit$. 
Put 
\begin{align*}
A^0_{ij} = P^0_{ij}+\sqrt{-1}Q^0_{ij}
\end{align*}
for $P^0_{ij},Q^0_{ij}\in C^\infty(U,\R)$. 
\begin{itemize}
\setlength{\parskip}{0cm}
\setlength{\itemsep}{0cm}
 \item[$({\rm i})$] $\frac{\del A^0_{ij}}{\del \theta^k}=\frac{\del A^0_{ik}}{\del \theta^j}$ hold for any 
$i,j,k$. 
 \item[$({\rm ii})$] Let ${\rm Ric}_{g_{J_s}}$ be 
the Ricci curvature of $g_{J_s}$. 
There exists a constant $\kappa\in\R$ 
such that 
${\rm Ric}_{g_{J_s}}\ge \kappa g_{J_s}$ hold for all $0<s<\delta$, if and only if 
$Q^0_{ij}(x,\theta)$ are independent of 
$\theta\in T^n$. 
\end{itemize}
\label{lower bdd of ricci}
\end{prop}
\begin{proof}
We have $\frac{\del A_{ij}}{\del \theta^k}
=s\frac{\del A^0_{ij}}{\del \theta^k}+
 \mathcal{O}_{C^1}(s^2)$ and 
$\frac{\del A_{ij}}{\del x_k}
=s\frac{\del A^0_{ij}}{\del x_k}+
 \mathcal{O}_{C^1}(s^2)$, 
then by \eqref{integrable} and taking 
$s\to 0$ we obtain $({\rm i})$. 

Next we show $({\rm ii})$. 
It suffices to discuss the existence of 
$\kappa$ such that 
$\rho_\omega\ge \kappa\omega$ holds. 
To show it, 
we write 
$\rho_{\omega}=
\sqrt{-1}\rho_{kl}\del \theta^k
\wedge\delb\theta^l$ 
for $\rho_{kl}\in\R$, 
then we expand 
$\rho_{kl}$ about $s=0$.

We have 
\begin{align*}
\det Q_{ij} &= s^n \left( \det Q^0_{ij} 
+ \mathcal{O}_{C^2}(s)\right),\\
\log \det Q_{ij} &= \log ( s^n) 
+ \log \det Q^0_{ij} 
+ \mathcal{O}_{C^2}(s),
\end{align*}
where $A^0_{ij}=P^0_{ij}+\sqrt{-1}Q^0_{ij}$, 
and 
\begin{align*}
Q^{ij} = s^{-1} Q^{0,ij} 
+ \mathcal{O}_{C^2}(1),
\end{align*}
where $(Q^{0,ij})_{i,j}$ is the inverse 
of $(Q^0_{ij})_{i,j}$. 
Since $\frac{\partial}{\partial\theta^i}
+A_{ij} \frac{\partial}{\partial x_j}$ forms 
the dual basis of $\del\theta^i$, we have 
\begin{align*}
\del\delb \log\det Q_{ij} 
&= \left( 
\frac{\partial^2(\log \det Q^0_{ij})}{\partial\theta^k\partial\theta^l}
+ \mathcal{O}_{C^0}(s)
\right)
\del\theta^k\wedge\delb\theta^l,\\
\del\alpha -\overline{\del\alpha}
&= \left( \mathcal{O}_{C^0}(s)\right)
\del\theta^k\wedge\delb\theta^l.
\end{align*}
Set $H=\log \det Q^0_{ij}$. 
We obtain 
\begin{align*}
\rho_\omega
&= \sqrt{-1}\left( 
\frac{\partial^2 H}{\partial\theta^k\partial\theta^l}
+ \mathcal{O}_{C^0}(s)
\right)
\del\theta^k\wedge\delb\theta^l. 
\end{align*}
Put $Q=(Q_{ij})_{ij}$, $Q^0=(Q^0_{ij})_{ij}$ 
and ${\rm Hess}_\theta H=(\frac{\partial^2 H}{\partial\theta^i\partial\theta^j})_{ij}$, 
and let $\sqrt{Q}$ be the symmetric matrix 
such that $\sqrt{Q}^2=Q$. 
Since $\omega=2\sqrt{-1}Q_{kl}\del\theta^k\wedge
\delb\theta^l$, 
then $\rho_\omega \ge \kappa\omega$ holds 
for some $\kappa\in\R$ 
if and only if the eigenvalues of 
$\sqrt{Q^{-1}}({\rm Hess}_\theta H+
\mathcal{O}_{C^0}(s))\sqrt{Q^{-1}}$ 
are bounded from the below by a constant. 
Since 
\begin{align*}
\sqrt{Q^{-1}} 
=\sqrt{s^{-1}}
\sqrt{(Q^0)^{-1} + \mathcal{O}_{C^2}(s)}
=\sqrt{s^{-1}}\left(
\sqrt{(Q^0)^{-1}} + \mathcal{O}_{C^2}(s)\right),
\end{align*}
we obtain 
\begin{align*}
&\quad\ \sqrt{Q^{-1}}({\rm Hess}_\theta H+
\mathcal{O}_{C^0}(s))\sqrt{Q^{-1}}\\
&= s^{-1}\left(
\sqrt{(Q^0)^{-1}} + \mathcal{O}_{C^2}(s)\right)
({\rm Hess}_\theta H+
\mathcal{O}_{C^0}(s))
\left(
\sqrt{(Q^0)^{-1}} + \mathcal{O}_{C^2}(s)\right)\\
&= s^{-1}
\sqrt{(Q^0)^{-1}}
{\rm Hess}_\theta H
\sqrt{(Q^0)^{-1}} + \mathcal{O}_{C^0}(1). 
\end{align*}

Therefore, the existence of 
the lower bound of the Ricci curvatures 
of $\{ g_{J_s}\}$ is equivalent to 
\begin{align*}
\sqrt{(Q^0)^{-1}}
{\rm Hess}_\theta H
\sqrt{(Q^0)^{-1}}\ge 0, 
\end{align*}
moreover, it is equivalent to 
${\rm Hess}_\theta H\ge 0$. 
Consequently, $H$ should be constant 
by the maximum principle.

By the imaginary part of $({\rm i})$, 
we can see that $Q^0_{ij}d\theta^j$ is a 
closed $1$-form on $\{ x\}\times T^n$, hence 
there exists a constant 
$\bar{Q}_{ij}$ depends only on $x$ such that 
$[\bar{Q}_{ij}d\theta^j]=[Q^0_{ij}d\theta^j]
\in H^1(\{ x\}\times T^n)$. 
Consequently, there are $F_i(x,\cdot)\in
C^\infty(\{ x\}\times T^n)$ such that 
$Q^0_{ij} = \bar{Q}_{ij} + \frac{\del F_i}{\del \theta^j}$ holds. 
Integrating this equality over 
$\{ x\}\times T^n$, 
we have 
\begin{align*}
\int_{\{ x\}\times T^n}
Q^0_{ij}(x,\theta)d\theta^1\cdots d\theta^n
=\bar{Q}_{ij}(x),
\end{align*}
which implies that $(\bar{Q}_{ij})_{i,j}$ is a 
positive definite symmetric matrix. 
Since $\frac{\del F_i}{\del \theta^j}
=\frac{\del F_j}{\del \theta^i}$ holds, 
one can see 
that $F_id\theta^i$ is a 
closed $1$-form on $\{ x\}\times T^n$, then by repeating the above argument, there are 
$F(x,\cdot)\in C^\infty(\{ x\}\times T^n)$ 
and $\bar{Q}_{i}(x)\in\R$ such that 
$F_i=\bar{Q}_{i} + \frac{\del F}{\del \theta^i}$, 
hence we may write 
\begin{align*}
Q^0_{ij} = \bar{Q}_{ij} 
+ \frac{\del^2 F}{\del \theta^i\del \theta^j}.
\end{align*}
Since $\bar{Q}_{ij}$ can be obtained by 
integrating $Q^0_{ij}$ along some cycles of 
$H_1(\{ x\}\times T^n,\Z)$, 
$(\bar{Q}_{ij})_{i,j}$ is also a positive definite 
symmetric matrix. 
Now we take another torus 
$T^n_{\rm copy}=\R^n/\Z^n$ and the coordinate 
$\tau^1,\ldots,\tau^n$ coming from $\R^n$. 
Next we regard 
$M_x:=\{ x\}\times T^n \times T^n_{\rm copy}$ 
as a complex manifold whose holomorphic coordinate is given by 
\begin{align*}
z^1:=\theta^1+\sqrt{-1}\tau^1,\ldots,
z^n:=\theta^n+\sqrt{-1}\tau^n.
\end{align*}
Define the K\"ahler form $\hat{\omega}_x$ 
on $M_x$ by 
$\hat{\omega}_x:=\sqrt{-1}\bar{Q}_{ij}(x)dz^i\wedge 
d\bar{z}^j$. 
Since $\bar{Q}_{ij}$ is constant on $M_x$, 
it is a Ricci-flat K\"ahler metric. 
Moreover 
\begin{align*}
\hat{\omega}_x + 4\sqrt{-1}\del\delb F
=\sqrt{-1}Q^0_{ij}(x,\theta)dz^i\wedge 
d\bar{z}^j
\end{align*}
is also a 
Ricci-flat K\"ahler metric 
since $\det Q^0$ is constant. 
By the uniqueness of the Ricci-flat K\"ahler metric 
in the fixed K\"ahler class, 
we obtain $Q^0_{ij} = \bar{Q}_{ij}$.
\end{proof}

\subsection{Convergence}\label{subsec conv}
Set 
\begin{align*}
U_{\Phi,r} &:= B_r\times (\R^n/{\rm Ker}\,\Phi)
=p_\Phi^{-1}(U_r),\\
S_r &:= S(L|_{U_r},h),\\
S_{\Phi,r} &:= U_{\Phi,r}\times S^1
= p_\Phi^{-1}(S_r)\\
\end{align*}
for $0<r\le R$.

For the brevity, put
\begin{align*}
\tilde{d}_A&:= 
\mbox{the Riemannian distance of }
\hat{g}_A\mbox{ on }S_{\Phi,R},\\
\hat{g}_J&:=\hat{g}(L,J,h,\sigma,\nabla),\\
d_J &:= 
\mbox{the Riemannian distance of }
\hat{g}_J
\mbox{ on }S(L,h),\\
d_A&:= 
\mbox{the Riemannian distance of }
\hat{g}_J|_{S(L|_U,h)}
\mbox{ on }S(L|_U,h),
\end{align*}
then 
\begin{align*}
d_A(p_\Phi(u),p_\Phi(v))
&= \inf_{k=0,1,\ldots,m-1}\tilde{d}_A(k\cdot u,v),\\
d_J(p_\Phi(u),p_\Phi(v)) &\le 
d_A(p_\Phi(u),p_\Phi(v))
\end{align*}
hold for all $u,v\in S_{\Phi,R}$.

Denote by $B_{g_J}(p,r)$ the 
geodesic ball in $(X,g_J)$ of radius $r$ centered at $p$, 
and denote by 
$B_{g_A}(p,r)$ the 
geodesic ball in $(U,g_A)$. 
Put 
\begin{align*}
\mathbf{0}:=(0,0)\in U,
\end{align*}
and 
\begin{align*}
B_{d_A}(r) 
&:= \{ p\in S(L|_U,h);\, 
d_A(p_\Phi(u_0),p)<r\},\\
B_{d_J}(r) 
&:= \{ u\in X;\, 
d_J(p_\Phi(u_0),u)<r\}.
\end{align*}

The the connection metric $\hat{g}_A$ 
given in Proposition 
\ref{description of connection metric} 
is written as 
\begin{align*}
\hat{g}_A =\sigma (dt-x_id\theta^i)^2
+ \Theta_{ij} d\theta^i d\theta^j
-2 P_{ik}Q^{jk} d\theta^i dx_j
+ Q^{ij} dx_i dx_j.
\end{align*}

\begin{prop}
\mbox{}

\begin{itemize}
\setlength{\parskip}{0cm}
\setlength{\itemsep}{0cm}
 \item[$({\rm i})$] $B_{g_J}\left(\mathbf{0}, \sqrt{\inf(\Theta^{-1})}R'
\right)\subset U_{R'}$ holds 
for any $0<R'\le\frac{R}{2}$.
 \item[$({\rm ii})$] Take $R_0>0$ such that 
\begin{align*}
2\left( 1+\frac{2\sqrt{\sup(\Theta^{-1})}}{\sqrt{\inf(\Theta^{-1})}}\right)R_0
+ \frac{\sqrt{n\sup\Theta} +2\sqrt{\sigma}\pi}{\sqrt{\inf(\Theta^{-1})}}
\le R.
\end{align*}
Then 
$d_J(p,p') = d_A(p,p')$ 
holds for any $p,p'\in S_{R_0}$. 
 \item[$({\rm iii})$] Assume that $\{ J_s\}_s$ satisfies $\spadesuit$. 
Then there are constants $s_0>0$, $0<R_0< \frac{R}{2}$ and $C>0$ such that 
\begin{align*}
B_{g_{J_s}}\left( \mathbf{0},\frac{CR'}{\sqrt{s}}
\right)\subset U_{R'}, \quad
d_{J_s}|_{S_{R_0}}= d_{A(s,\cdot)}|_{S_{R_0}}
\end{align*}
hold for 
any $0<s\le s_0$ and $0<R'\le \frac{R}{2}$. 
\end{itemize}
\label{metric ball and Euclidean ball}
\end{prop}
\begin{proof}
$({\rm i})$ 
Let $p\in B_{g_J}\left(\mathbf{0}, \sqrt{\inf(\Theta^{-1})}R'
\right)$ and 
suppose $p\notin U_{R'}$. 
Then there is a piecewise smooth path $c_1\colon [0,1]\to X$ 
such that $c_1(0)=\mathbf{0}$, $c_1(1)=p$ and the length 
$\mathcal{L}(c_1)$ is less than 
$\sqrt{\inf(\Theta^{-1})}R'$. 
Let 
\begin{align*}
\tau_1:=\inf\{ \tau\in [0,1];\, c_1(\tau)
\notin U_{R'}\} \le 1.
\end{align*}
Then by the first inequality of Proposition 
\ref{estimate on distance}, 
\begin{align*}
\mathcal{L}(c_1)\ge \mathcal{L}(c_1|_{[0,\tau_1]})
\ge d_{g_J}(\mathbf{0},c_1(\tau_1)) \ge \sqrt{\inf(\Theta^{-1})}R'
\end{align*}
holds, hence we have the contradiction. 

$({\rm ii})$ 
Take 
$R_0>0$ which satisfies the assumption. 
Let $p,p'\in S_{R_0}$ and 
suppose $d_J(p,p') < d_A(p,p')$. 
Then there is a piecewise smooth path 
$c_2\colon [0,1]\to S(L,h)$ connecting $p$ and $p'$ 
such that ${\rm Im}(c_2)$ is not contained in $S_{\frac{R}{2}}$ 
and $\mathcal{L}(c_2)$ is less than 
$d_A(p,p')$. 
Put 
\begin{align*}
\tau_2:=\inf\{ \tau\in [0,1];\, c_2(\tau)
\notin S_{\frac{R}{2}}\}, 
\end{align*}
then 
\begin{align*}
\mathcal{L}(c_2)\ge \mathcal{L}(c_2|_{[0,\tau_2]})
\ge d_J(c_2(0),c_2(\tau_2)) 
\end{align*}
holds. 
Since $\pi\colon (S(L,h),\hat{g}_J)\to 
(X,g_J)$ is a Riemannian submersion, 
\begin{align*}
d_J(c_2(0),c_2(\tau_2)) 
\ge d_{g_J}(\pi(c_2(0)),\pi(c_2(\tau_2))) 
\end{align*}
holds, then we can see 
\begin{align*}
\mathcal{L}(c_2)
\ge d_{g_J}(\pi(c_2(0)),\pi(c_2(\tau_2)))
\ge \sqrt{\inf(\Theta^{-1})}
\left( \frac{R}{2} - R_0\right),
\end{align*}
by the first inequality of 
Proposition \ref{estimate on distance}.
The second inequality of 
Proposition \ref{estimate on distance} gives 
\begin{align*}
\sqrt{\inf(\Theta^{-1})}
\left( \frac{R}{2} - R_0\right) &< 
d_A(p,p') \\
&\le 
2\sqrt{\sup (\Theta^{-1})}R_0
+ \frac{\sqrt{n\sup \Theta}}{2} + 
\sqrt{\sigma} \pi,
\end{align*}
therefore we obtain 
\begin{align*}
\frac{R}{2} 
< \left( 1+  
\frac{2\sqrt{\sup (\Theta^{-1})}}
{\sqrt{\inf(\Theta^{-1})}}\right) R_0
 + \frac{1}{\sqrt{\inf(\Theta^{-1})}}
\left( \frac{\sqrt{n\sup \Theta}}{2}  + \sqrt{\sigma}\pi\right),
\end{align*}
which contradicts the assumption.  

$({\rm iii})$ Since we have 
\begin{align*}
\sqrt{\inf(\Theta^{-1})}
&= \frac{1}{\sqrt{s}}\left( \sqrt{\inf((\Theta^0)^{-1})}+\mathcal{O}(s)\right),\\
\sqrt{\sup(\Theta^{-1})}
&= \frac{1}{\sqrt{s}}\left( \sqrt{\sup((\Theta^0)^{-1})}+\mathcal{O}(s)\right),\\
\sqrt{\sup(\Theta)}
&= \sqrt{s}\left( \sqrt{\sup (\Theta^0)}+\mathcal{O}(s)\right)
\end{align*}
by the Hoffman-Wielandt's inequality 
\cite{HoffmanWielandt1953}, 
there exists $s_0>0$ such that 
\begin{align*}
\frac{2\sqrt{\sup(\Theta^{-1})}}{\sqrt{\inf(\Theta^{-1})}}
&\le \frac{3\sqrt{\sup((\Theta^0)^{-1})}}{\sqrt{\inf((\Theta^0)^{-1})}},\\
\frac{\sqrt{n\sup\Theta} +2\sqrt{\sigma}\pi}{\sqrt{\inf(\Theta^{-1})}}
&\le 
\frac{R}{10}
\end{align*}
for all $s\le s_0$. 
If we take $0<R_0<\frac{R}{2}$ such that 
\begin{align*}
2\left( 1+\frac{3\sqrt{\sup((\Theta^0)^{-1})}}{\sqrt{\inf((\Theta^0)^{-1})}}\right) R_0
\le \frac{9R}{10},
\end{align*}
then the assumption of $({\rm ii})$ is satisfied for $s\le s_0$, 
hence we have $d_{J_s}|_{S_{R_0}}= d_{A(s,\cdot)}|_{S_{R_0}}$. 
Moreover, if we put 
\begin{align*}
C:=\inf_{0<s\le s_0} 
\sqrt{s\inf(\Theta^{-1})}>0,
\end{align*}
then we can see
\begin{align*}
\sqrt{\inf(\Theta^{-1})}R'\ge \frac{CR'}{\sqrt{s}}
\end{align*}
for $R'\le \frac{R}{2}$, hence 
we have 
$B_{g_{J_s}}\left( \mathbf{0},\frac{CR'}{\sqrt{s}}
\right)\subset U_{R'}$ by $({\rm i})$. 
\end{proof}

Next we consider $\omega$-compatible 
complex structures $J,J'$, and 
compare the Riemannian 
distances of $g_J$ and $g_{J'}$. 
We will show that if $g_J$ and $g_{J'}$ 
are close to each other in some sense 
then their Riemannian 
distances are also close to each other.

Now, we define 
the distance $d_{{\rm Sym}^+(\R^N)}$ on 
\begin{align*}
{\rm Sym}^+(\R^N):= 
\{ g\in M_N(\R);\, g_{ij}=g_{ji},\, g>0\}
\end{align*}
as follows. 
For $g\in {\rm Sym}^+(\R^N)$, take 
$v_1,\ldots,v_N\in\R^N$ such that 
$g(v_i,v_j)=\delta_{ij}$. 
For $g'\in {\rm Sym}^+(\R^N)$ let 
$\lambda_1,\ldots,\lambda_N\in\R$ be 
eigenvalues of $(g'(v_i,v_j))_{i,j}$. 
Then define 
\begin{align*}
d_{{\rm Sym}^+(\R^N)}(g,g')
:= \max_{i}| \log \lambda_i |. 
\end{align*}
Moreover, if $g,g'$ are Riemannian metrics 
on $M$, then define 
\begin{align*}
d_{{\rm Sym}^+(M)}(g,g')
:= \sup_{x\in M} d_{{\rm Sym}^+(T_xM)}(g_x,g_x').
\end{align*}

\begin{lem}
Let $M$ be a smooth manifold of dimension $N$, 
$g,g'$ be Riemannian metrics on $M$ and 
$d,d'$ be the Riemannian distances of $g,g'$, 
respectively. 
If we assume 
$d_{{\rm Sym}^+(M)}(g,g')\le 2\log 2$,  
then 
\begin{align*}
| d(p_0,p_1) - d'(p_0,p_1) |
\le d_{{\rm Sym}^+(M)}(g,g')
d'(p_0,p_1) 
\end{align*}
holds. 
Moreover, for any $f\in C_0(M)$ 
\begin{align*}
\left| \int_M f d\mu_g - \int_M f d\mu_{g'}\right|
\le N \sup |f|\cdot \mu_{g'}({\rm supp}(f))\cdot d_{{\rm Sym}^+(M)}(g,g')
\end{align*}
holds 
if $d_{{\rm Sym}^+(M)}(g,g')\le \frac{\log 2}{N}$.
\label{estimate on distance 2}
\end{lem}
\begin{proof}
Let $c\colon [a,b]\to M$ be a 
piecewise smooth path, and 
denote by $\mathcal{L}_g(c)$ be the length 
of $c$ with respect to $g$. 
Since we have 
\begin{align*}
g(c'(t),c'(t)) &\le \exp\left( 
d_{{\rm Sym}^+(T_{c(t)}M)}(g_{c(t)},g'_{c(t)})\right)g'(c'(t),c'(t))\\
&\le \exp\left( 
d_{{\rm Sym}^+(M)}(g,g')\right)g'(c'(t),c'(t))
\end{align*}
then we can see 
\begin{align*}
\mathcal{L}_g(c) \le 
\exp\left( 
\frac{d_{{\rm Sym}^+(M)}(g,g')}{2}\right)\mathcal{L}_{g'}(c)
\end{align*}
and 
\begin{align*}
d(p_0,p_1) \le 
\exp\left( 
\frac{d_{{\rm Sym}^+(M)}(g,g')}{2}\right)
d'(p_0,p_1).
\end{align*}
By the symmetry we also have 
\begin{align*}
\exp\left( 
-\frac{d_{{\rm Sym}^+(M)}(g,g')}{2}\right)
d'(p_0,p_1) \le d(p_0,p_1).
\end{align*}
Therefore, we obtain 
\begin{align*}
d'(p_0,p_1)-d(p_0,p_1) \le \left( 1
- \exp\left( 
- \frac{d_{{\rm Sym}^+(M)}(g,g')}{2}\right)\right)
d'(p_0,p_1)
\end{align*}
and 
\begin{align*}
d(p_0,p_1)-d'(p_0,p_1) \le \left( \exp\left( 
\frac{d_{{\rm Sym}^+(M)}(g,g')}{2}\right)
- 1 \right)
d'(p_0,p_1).
\end{align*}
Since $1-e^{-\frac{t}{2}}\le t$ and 
$e^{\frac{t}{2}}-1\le t$ holds for any 
$0\le t\le 2\log 2$, we have the first inequality. 

Next we take $f\in C_0(M)$ and 
denote by $d\mu_g$ the Riemannian measure of $g$. 
Then we have 
\begin{align*}
\left| \int_M f d\mu_g - \int_M f d\mu_{g'}\right|
\le \int_M |f| \left| \frac{\det{g}}{\det{g'}} - 1 \right|d\mu_{g'}. 
\end{align*}
Since $|\log \frac{\det{g}}{\det{g'}} | \le Nd_{{\rm Sym}^+(M)}(g,g')$ holds 
and $|e^t-1|\le 2|t|$ holds for $|t|\le \log 2$, 
we can see 
\begin{align*}
\left| \int_M f d\mu_g - \int_M f d\mu_{g'}\right|
\le N \sup |f|\cdot \mu_{g'}({\rm supp}(f))\cdot d_{{\rm Sym}^+(M)}(g,g')
\end{align*}
if $d_{{\rm Sym}^+(M)}(g,g')\le \frac{\log 2}{N}$.
\end{proof}

\begin{lem}
Let $g,g'\in {\rm Sym}^+(\R^N)$ and 
$\{ v_1,\cdots,v_N\}$ be a basis of $\R^N$. 
Put $\mathbf{g}=(g(v_i,v_j))_{i,j}$ and 
$\mathbf{g}'=(g'(v_i,v_j))_{i,j}$. 
Denote by $\alpha_1,\ldots,\alpha_N$ be 
the eigenvalues of $\mathbf{g}'\mathbf{g}^{-1}$. 
Then $\alpha_i\in\R$ and 
$d_{{\rm Sym}^+(\R^N)}(g,g')
=\max_i |\log \alpha_i |$. 
\label{eigen distance}
\end{lem}
\begin{proof}
Let $\sqrt{\mathbf{g}}$ be the square root 
of $\mathbf{g}$. 
If we put 
$e_i=\sum_j{\sqrt{\mathbf{g}}^{-1}}_{ij}v_j$, 
then $e_1,\cdots,e_N$ is an orthonormal basis 
of $(\R^N,g)$, therefore we have 
\begin{align*}
d_{{\rm Sym}^+(\R^N)}(g,g') 
= \max_i|\log \lambda_i|,
\end{align*}
where 
$\lambda_i$ are the eigenvalues of 
\begin{align*}
(g'(e_i,e_j))_{ij}=\sqrt{\mathbf{g}}^{-1}\mathbf{g}'\sqrt{\mathbf{g}}^{-1}.
\end{align*}
Since we have 
\begin{align*}
\sqrt{\mathbf{g}}^{-1}\cdot 
\left( \mathbf{g}'\mathbf{g}^{-1}\right)\cdot 
\sqrt{\mathbf{g}}
= \sqrt{\mathbf{g}}^{-1}\mathbf{g}'\sqrt{\mathbf{g}}^{-1}, 
\end{align*}
$\{ \alpha_1,\ldots,\alpha_N\}=
\{ \lambda_1,\ldots,\lambda_N\}$ holds. 
\end{proof}

Suppose that $X_0$ 
is a strict $m$-BS fiber and fix a small $s>0$ and a frame 
\begin{align*}
dt-x_id\theta^i,\,\sqrt{s}d\theta^1,
\ldots,\sqrt{s}d\theta^n,\,
\frac{1}{\sqrt{s}}dx_1,\ldots,\frac{1}{\sqrt{s}}dx_n
\end{align*}
of $T^*(U_\Phi\times S^1)$. 
Then the matrix representation 
of $\hat{g}_A$ 
is given by 
\[ \mathbf{g}_A:=
\left (
\begin{array}{ccc}
\sigma & 0 & 0 \\
0 & s^{-1}\Theta & -PQ^{-1} \\
0 & -Q^{-1}P & s Q^{-1}
\end{array}
\right ), 
\]
and its inverse is 
\[ \mathbf{g}^{-1}_A=
\left (
\begin{array}{ccc}
\sigma^{-1} & 0 & 0 \\
0 & s Q^{-1} & Q^{-1}P \\
0 & PQ^{-1} & s^{-1}\Theta
\end{array}
\right ).
\]

Suppose that $\{ A(s,\cdot)\}_s$ 
corresponds to $\{ J_s\}$ which satisfies 
$\spadesuit$. 
Fix $r\ge 1$. 
Then there is a constant $K>0$ 
depending only on $\{ A(s,\cdot)\}_s$ 
such that 
\begin{align*}
| A(s,x,\theta) - sA^0(x,\theta)| &\le 
Ks^2\\
| A^0(x,\theta) - A^0(0,\theta)|&\le K\| x\|
\end{align*}
for any $(x,\theta)\in U_\Phi$. 
If $(x,\theta)\in U_{\Phi,\sqrt{s}r}$ 
for $r\ge 1$ and $s>0$ with $\sqrt{s}r\le R$, 
then we have $s\le \sqrt{s}\frac{R}{r}\le
\sqrt{s}r$ since $R\le 1$, hence we obtain 
\begin{align*}
\left| s^{-1}A(s,x,\theta) - A^0(0,\theta)\right| &\le 
K\sqrt{s}r.
\end{align*}
Here we write 
\begin{align*}
f_0(s,x,\theta) \cong_{\sqrt{s}r} f_1(s,x,\theta)
\end{align*}
if there is a constant $K>0$ such that 
$|f_0(s,x,\theta) - f_1(s,x,\theta)|\le 
K\sqrt{s}r$ holds for any $(x,\theta)\in U_{\Phi,\sqrt{s}r}$.

Now $A'(s,x,\theta):=sA^0(0,\theta)$ gives 
another family of complex structures 
$\{ J'_s\}_s$ which satisfies $\spadesuit$, 
by $({\rm i})$ of 
Proposition \ref{lower bdd of ricci}. 
Since we have  
\begin{align*}
s^{-1}\Theta 
&\cong_{\sqrt{s}r} \Theta^0(0,\theta),\\
PQ^{-1}
&\cong_{\sqrt{s}r} P^0(0,\theta)Q^0(0,\theta)^{-1},\\
Q^{-1}P
&\cong_{\sqrt{s}r} Q^0(0,\theta)^{-1}P^0(0,\theta),\\
sQ^{-1}
&\cong_{\sqrt{s}r} Q^0(0,\theta)^{-1},
\end{align*}
where $\Theta^0(0,\theta)=Q^0(0,\theta)+
P^0(0,\theta)Q^0(0,\theta)^{-1}P^0(0,\theta)$, 
then we obtain 
\begin{align*}
\mathbf{g}^{-1}_{A'}\mathbf{g}_A
\cong_{\sqrt{s}r} I_{2n+1}.
\end{align*}
By Lemma \ref{eigen distance}, 
the eigenvalues of $\mathbf{g}^{-1}_{A'}
\mathbf{g}_A$ are real.
If $1+\lambda\in\R$ is one of the 
eigenvalues, then 
\begin{align*}
f(\lambda) := \det\left( (1+\lambda) I_{2n+1} 
- \mathbf{g}^{-1}_{A'}\mathbf{g}_A\right)
=0
\end{align*}
holds. 
Since we have  
\begin{align*}
f(\lambda) = \det\left\{ \lambda I_{2n+1} 
+ (I_{2n+1} - \mathbf{g}^{-1}_{A'}\mathbf{g}_A)
\right\}, 
\end{align*}
there exists a constant 
$K>0$ depending only on 
$\{ A(s,\cdot)\}_s$, and there 
exist $c_0,c_1,\ldots,c_{2n}\in\R$ 
such that $\max_i|c_i|\le K$ 
and 
\begin{align*}
f(\lambda) = \lambda^{2n+1}
+ \sum_{i=0}^{2n}c_i (\sqrt{s}r)^{2n+1-i}\lambda^{i}.
\end{align*}

\begin{lem}
For any $n\in\Z_{\ge 0}$, 
$K>0$ and $r\ge 1$ there is a sufficiently large 
$N>0$ depending only on $n$ and $K$ 
such that for any 
$c_0,c_1,\ldots,c_{2n}\in[-K,K]$ and 
$\varepsilon>0$, 
the solution $\lambda$ of the equation 
\begin{align*}
f(\lambda) = \lambda^{2n+1}
+ \sum_{i=0}^{2n}c_i \varepsilon^{2n+1-i}\lambda^{i}
= 0
\end{align*}
always satisfies $|\lambda|\le N\varepsilon$. 
\label{polynomial}
\end{lem}
\begin{proof}
Put $\lambda=\varepsilon t$. 
Then $f(\lambda)=\varepsilon^{2n+1}\left( t^{2n+1}
+ \sum_{i=0}^{2n}c_i t^i\right)$. 
If $f(\lambda) = 0$ then 
we have $t^n=- \sum_{i=0}^{2n}c_i t^i$. 
Suppose $|t|\ge1$. 
Then 
\begin{align*}
|t|^{2n+1}\le \sum_{i=0}^{2n}|c_i||t|^i
\le \sum_{i=0}^{2n}K|t|^{2n}=(2n+1)K|t|^{2n}
\end{align*}
holds, hence $|t|\le (2n+1)K$ is obtained. 
Consequently we can see
$|\lambda|\le \max\{1,(2n+1)K\}\varepsilon$. 
\end{proof}

By Lemma \ref{polynomial} 
we can see $|\log (1+\lambda)|
\le N\sqrt{s}r$ 
for the eigenvalue $1+\lambda$ of 
$\mathbf{g}^{-1}_{A'}\mathbf{g}_A$, 
where $N$ is the constant depending only on 
$K$. 
Therefore, 
we  obtain the following proposition 
by Lemma \ref{eigen distance}. 
\begin{prop}
Let $A,A'$ be as above and 
let $r\ge 1$, $s>0$ with $\sqrt{s}r\le R$. 
Then there exists a constant $C>0$ 
depending only on $A$ such that 
\begin{align*}
d_{{\rm Sym}^+(U_{\Phi,\sqrt{s}r}\times S^1)}
(\hat{g}_{A},\hat{g}_{A'}) \le C\sqrt{s}r.
\end{align*}
\label{estimate on distance 3}
\end{prop}

From now on we assume $R>0$ satisfies 
\begin{align*}
CR\le 2\log 2,
\end{align*}
where $C$ is the constant in Proposition 
\ref{estimate on distance 3}. 
Then Lemma \ref{estimate on distance 2} 
holds for 
\begin{align*}
M=S_{\Phi,R},\quad g=\hat{g}_A,\quad g'=\hat{g}_{A'}. 
\end{align*}
and for 
\begin{align*}
M=S_{\Phi,R},\quad g=\hat{g}_{A'},\quad g'=\hat{g}_A. 
\end{align*}

\begin{prop}
Let $\{ J_s\}_s$ satisfy $\spadesuit$ 
and $A'(s,x,\theta):=sA^0(0,\theta)$. 
\begin{itemize}
\setlength{\parskip}{0cm}
\setlength{\itemsep}{0cm}
 \item[$({\rm i})$] There are positive constants 
$C_0',C_1'$ depending only on $A^0(0,\cdot)$ and $\sigma$ 
such that 
\begin{align*}
B_{g_{A'}}(\mathbf{0},C_0'r)
\subset U_{\sqrt{s}r} \subset 
B_{g_{A'}}(\mathbf{0},C_1'r)
\end{align*}
for any $r\ge 1$ and $s>0$ with 
$\sqrt{s}r\le R$. 
 \item[$({\rm ii})$] Suppose 
$X_0$ is a strict $m$-BS fiber. Then there are 
constants $C>0$ and $0<R_0< \frac{R}{2}$ 
depending only on 
$A$ and $\sigma$ such that  
\begin{align*}
|d_{J_s}(p,q)-d_{A'}(p,q)| < C\sqrt{s}r^2
\end{align*}
holds for any $r\ge 1$, $s>0$ with 
$\sqrt{s}r\le R_0$ 
and $p,q\in S_{\sqrt{s}r}$. 
 \item[$({\rm iii})$] There are positive constants 
$C_0,C_1$ and $0<R_0 < \frac{R}{2}$ depending only on $A$ and $\sigma$  
such that 
\begin{align*}
B_{g_{J_s}}(\mathbf{0},C_0r)
\subset U_{\sqrt{s}r} \subset 
B_{g_{J_s}}(\mathbf{0},C_1r)
\end{align*}
for any $r\ge 1$, $s>0$ with $\sqrt{s}r\le R_0$. 
\end{itemize}
\label{estimate on distance 4}
\end{prop}
\begin{proof}
$({\rm i})$ Apply Proposition \ref{estimate on distance} 
for $A'$. 
Then there are positive constants $C_2,C_3,C_4$ 
depending only on $A^0(0,\cdot)$ and $\sigma$ 
such that 
\begin{align*}
C_2\sqrt{s}^{-1}\| x\| \le d_{g_{A'}}(\mathbf{0},u) \le C_3\sqrt{s}^{-1}\| x\| + C_4
\end{align*}
for any $u=(x,\theta)$ and $s>0$. 
If $\| x\|<\sqrt{s}r$ then 
\begin{align*}
d_{g_{A'}}(\mathbf{0},u) < C_3 r + C_4 \le (C_3+C_4)r
\end{align*}
holds since $r\ge 1$, which implies 
$U_{\sqrt{s}r}\subset B_{g_{A'}}(\mathbf{0},(C_3+C_4)r)$. 
On the other hand if 
$d_{g_{A'}}(\mathbf{0},u) < C_2r$ holds then 
\begin{align*}
C_2\sqrt{s}^{-1}\| x\| \le d_{g_{A'}}(\mathbf{0},u) <C_2r
\end{align*}
gives $\| x\|<\sqrt{s} r$, hence 
$B_{g_{A'}}(\mathbf{0},C_2r)\subset U_{\sqrt{s}r}$ holds. 

$({\rm ii})$ By applying 
Proposition \ref{estimate on distance 3}, 
there is a constant $C_5>0$ such that 
\begin{align*}
d_{{\rm Sym}^+(S_{\Phi,\sqrt{s}r})}
(\hat{g}_{A},\hat{g}_{A'}) \le C_5\sqrt{s}r
\end{align*}
holds if $\sqrt{s}r\le R$. 
Now take $R_0\le \min\{ \frac{2\log 2}{C_5},R\}$ 
and assume $\sqrt{s}r\le R_0$, 
then we may apply 
Lemma \ref{estimate on distance 2} and we have 
\begin{align*}
| \tilde{d}_A(u,v) - \tilde{d}_{A'}(u,v) |
\le d_{{\rm Sym}^+(M)}(g,g')
\tilde{d}_{A'}(u,v) 
\le C_5\sqrt{s}r \tilde{d}_{A'}(u,v)
\end{align*}
for all $u,v\in S_{\Phi,\sqrt{s}r}$. 
By the same argument in the proof of 
Proposition \ref{estimate on distance}, 
we have the upper estimate 
\begin{align*}
\tilde{d}_{A'}(u,v)
&\le \sqrt{\sup (Q^0)^{-1}}
\frac{\| x-x'\|}{\sqrt{s}}
+\sqrt{s}\sqrt{\sigma r^2+\cdot\sup\Theta^0}\cdot {\rm diam} (\R^n/{\rm Ker}\Phi)\\
&\quad\quad +\sqrt{\sigma}\pi,
\end{align*}
where $u=(x,\theta,e^{\sqrt{-1} t}), 
v=(x',\theta',e^{\sqrt{-1} t'})$. 
Since $\| x-x'\|\le 2\sqrt{s}r$, 
$\sqrt{s}r\le R\le 1$ and $r\ge 1$, 
we have $\frac{\| x-x'\|}{\sqrt{s}}\le 2r$ 
and $s\le \frac{R^2}{r^2}\le 1$, 
then there is a constant $C_6>0$ 
depending only on $A^0,\sigma,\Phi$ 
such that $\tilde{d}_{A'}(u,v)\le 
C_6 r$, 
which gives 
\begin{align*}
| \tilde{d}_A(u,v) - \tilde{d}_{A'}(u,v) |
< 2C_5C_6\sqrt{s}r^2. 
\end{align*}
Therefore, we can see
\begin{align*}
d_A(p_\Phi(u),p_\Phi(v)) &= \inf_{k=0,1,\ldots,m-1}
\tilde{d}_A(k\cdot u,v)\\
&< \inf_{k=0,1,\ldots,m-1}\left\{ 
\tilde{d}_{A'}(k\cdot u,v) + 2C_5C_6\sqrt{s}r^2
\right\}\\
&=d_{A'}(k\cdot u,v) + 2C_5C_6\sqrt{s}r^2
\end{align*}
and similarly $d_{A'}(p_\Phi(u),p_\Phi(v)) <d_A(k\cdot u,v) + 2C_5C_6\sqrt{s}r^2$ is obtained. 
By $({\rm iii})$ of Proposition 
\ref{metric ball and Euclidean ball}, 
we can take $0<R_0'<\frac{R}{2}$ and 
$s_0>0$ such that 
$d_{J_s}|_{S_{R_0'}}=d_A|_{S_{R_0'}}$ holds 
for any $0<s\le s_0$. 
If we put $C= 2C_5C_6$ and 
$R_0=\min\{ \frac{2\log 2}{C_5},R_0',\sqrt{s_0}\}$, 
then $\sqrt{s}r\le R_0$ implies $s\le s_0$, 
hence 
we have $({\rm ii})$. 

$({\rm iii})$ 
Take $C,s_0,R_0$ as in $({\rm iii})$ of 
Proposition 
\ref{metric ball and Euclidean ball} and 
replace $R_0$ by the smaller one such that 
$R_0\le \sqrt{s_0}$. 
Then we have 
$B_{g_{J_s}}(\mathbf{0},Cr) \subset U_{\sqrt{s}r}$ 
if $\sqrt{s}r\le R_0$. 
Next we assume $u\in U_{\sqrt{s}r}$. 
By $({\rm i})$, we have 
$u\in B_{g_{A'}}(\mathbf{0},C_1'r)$. 
Since $\pi\colon (S_R,\hat{g}_{J_s}) \to 
(U_R, g_{J_s})$ and 
$\pi\colon (S_R,\hat{g}_{A'}) \to 
(U_R, g_{A'})$ are Riemannian submersions, 
therefore $({\rm ii})$ gives 
\begin{align*}
d_{g_{J_s}}(\pi(u),\pi(u'))
&=\inf_{e^{\sqrt{-1}t}\in S^1}
d_{J_s}(ue^{\sqrt{-1}t},u')\\
&\le \inf_{e^{\sqrt{-1}t}\in S^1}
d_{A'}(ue^{\sqrt{-1}t},u')+C\sqrt{s}r^2\\
&= d_{g_{A'}}(\pi(u),\pi(u'))+C\sqrt{s}r^2. 
\end{align*}
Consequently we obtain 
\begin{align*}
d_{g_{J_s}}(\mathbf{0},u)
\le d_{g_{A'}}(\mathbf{0},u)+C\sqrt{s}r^2
< C_1'r+CR_0r, 
\end{align*}
which implies 
$U_{\sqrt{s}r}\subset B_{g_{J_s}}(\mathbf{0},C_1r)$ 
by putting $C_1=C_1'+C$. 
\end{proof}

\begin{prop}
Let $\{ J_s\}_s$ satisfy $\spadesuit$ 
and 
$A'(s,x,\theta):=sA^0(0,\theta)$. 
There exist constants $R_0,C>0$ such that 
\begin{align*}
{\rm id}\colon \left( 
\pi^{-1}(B_{g_{J_s}}(\mathbf{0},r-C\sqrt{s}r^2)),d_{J_s}\right) \to 
\left( \pi^{-1}(B_{g_{A'(s,\cdot)}}(\mathbf{0},r)), d_{A'(s,\cdot)}\right)
\end{align*}
is a Borel 
$C\sqrt{s}r^2$-$S^1$-equivariant 
Hausdorff approximation 
for any $r\ge 1$ and $s\le \frac{R_0^2}{Cr^2}$. 
Moreover, if $f\colon S_R\to \R$ is 
a Borel function such that 
${\rm supp}(f)\subset \overline{S_{R'}}$ for 
some $R'\le R$ and $\sup|f|<\infty$, then 
\begin{align*}
\left| \int_{S_R} 
f d\mu_{\hat{g}_{A'}}
- \int_{S_R} 
f d\mu_{\hat{g}_{J_s}} \right|
&\le C \sup|f| (R')^{n+1}
\end{align*}
holds. 
\label{approximation}
\end{prop}
\begin{proof}
Fix $r\ge 1$. 
Take $R_0,C_0,C_1,C_0',C_1',C$ such that 
Proposition \ref{estimate on distance 4} 
holds. 
We may suppose $C>1$ and 
$C_0=C_0'=C^{-1}$, $C_1=C_1'=C$. 
Then $B_{g_{J_s}}(\mathbf{0},r)\subset U_{C\sqrt{s}r}$ 
and 
\begin{align*}
|d_{J_s}(p,q)-d_{A'}(p,q)|<C^2\sqrt{s}r^2
\end{align*}
hold for any 
$p,q\in S_{C\sqrt{s}r}$ and 
$0<s\le \frac{R_0^2}{C^2r^2}$. 
If $u\in B_{g_{J_s}}(\mathbf{0},r-C^2\sqrt{s}r^2)$, 
then 
\begin{align*}
d_{g_{A'}}(\mathbf{0}, u)<  
d_{g_{J_s}}(\mathbf{0}, u) + C^2\sqrt{s}r^2
<r,
\end{align*} 
which implies 
$B_{g_{J_s}}(\mathbf{0},r-C^2\sqrt{s}r^2)\subset B_{g_{A'}}(\mathbf{0},r)$. 

Now, $B_{g_{A'}}(\mathbf{0},r)
\subset U_{C\sqrt{s}r}$ 
holds. 
We also have
\begin{align*}
B_{g_{A'}}(\mathbf{0},r)
\subset B_{g_{J_s}}(\mathbf{0},r+C^2\sqrt{s}r^2).
\end{align*}
Since $d_{g_{J_s}}$ is an intrinsic metric, we have 
\begin{align*}
B_{g_{J_s}}(\mathbf{0},r+C^2\sqrt{s}r^2 )
=B_{g_{J_s}}\left( B_{g_{J_s}}(\mathbf{0},r-C^2\sqrt{s}r^2),2C^2\sqrt{s}r^2 \right).
\end{align*}
hence we can see
that 
\begin{align*}
{\rm id}\colon \left( 
\pi^{-1}(B_{g_{J_s}}(\mathbf{0},r-C^2\sqrt{s}r^2)),d_{J_s}\right) \to 
\left( \pi^{-1}(B_{g_{A'}}(\mathbf{0},r)), d_{A'(s,\cdot)}\right)
\end{align*}
is a Borel 
$\varepsilon_i$-$S^1$-equivariant 
Hausdorff approximation. 

Let 
$f\colon S_R \to \R$ be a Borel function 
such that ${\rm supp}(f)\subset \overline{S_{R'}}$ 
for some $R'\le R$ 
and $\sup|f|<\infty$. 
Then one can see 
\begin{align*}
\left| \int_{S_R} 
f d\mu_{\hat{g}_{A'}}
- \int_{S_R} 
f d\mu_{\hat{g}_{J_s}} \right|
&\le 2n\sup|f|\cdot \mu_{\hat{g}_{A'}}(S_{R'})
\cdot CR'
\end{align*}
by Lemma \ref{estimate on distance 2} and Proposition 
\ref{estimate on distance 3}. 
Since 
\begin{align*}
d\mu_{g_{A'}} =
\det(\mathbf{g}_{A'}) dt d\theta^1\cdots 
d\theta^n dx_1\cdots dx_n
\end{align*}
and 
\[ \det(\mathbf{g}_{A'})=\sigma\det
\left (
\begin{array}{cc}
\Theta^0 & -P^0(Q^0)^{-1} \\
-(Q^0)^{-1}P^0 & (Q^0)^{-1}
\end{array}
\right ), 
\]
one can see that 
$\mu_{g_{A'}}(S_{R'}) = \sigma C'(R')^n$, 
which gives the assertion. 
\end{proof}

Let $\{ J_s\}_s$ satisfy $\spadesuit$ 
and $A'(s,x,\theta):=sA^0(0,\theta)$. 
By Proposition \ref{lower bdd of ricci}, 
$P^0_{ij}(0,\theta)d\theta^j$ is a 
closed $1$-form on $T^n$. 
Then there are constants $\bar{P}_{ij}\in \R$ 
such that 
\begin{align*}
[P^0_{ij}(0,\cdot)d\theta^j]
=[\bar{P}_{ij}d\theta^j]\in H^1(T^n,\R),
\end{align*}
hence there are $\mathcal{H}_i\in C^\infty(T^n)$ such that 
\begin{align*}
P^0_{ij}(0,\cdot)d\theta^j
=\bar{P}_{ij}d\theta^j + d \mathcal{H}_i.
\end{align*}
Since $P^0_{ij}=P^0_{ji}$ and 
\begin{align*}
\bar{P}_{ij} = \int_{T^n}
P^0_{ij}(0,\theta)d\theta^1\cdots d\theta^n,
\end{align*}
we have $\bar{P}_{ij}=\bar{P}_{ji}$ and 
$\frac{\del \mathcal{H}_i}{\del\theta^j}
=\frac{\del \mathcal{H}_j}{\del\theta^i}$. 
Consequently, $\mathcal{H}_i d\theta^i$ is 
closed, therefore there are $\bar{P}_i\in \R$ 
and $\mathcal{H}\in C^\infty(T^n)$ such that 
$\mathcal{H}_i=\bar{P}_i+\frac{\del \mathcal{H}}{\del\theta^i}$, 
which gives 
\begin{align*}
P^0_{ij}(0,\cdot)
=\bar{P}_{ij} 
+ \frac{\del^2 \mathcal{H}}{\del\theta^i \del\theta^j}.
\end{align*}
If ${\rm Ric}_{g_{J_s}}$ has the lower bound, 
then by Proposition \ref{lower bdd of ricci} 
we have $Q^0_{ij}(0,\cdot)=\bar{Q}_{ij}\in\R$ 
and 
\begin{align*}
g_{A'} &= s(\bar{Q}_{ij} + P^0_{ik}\bar{Q}^{kl}P^0_{lj}) d\theta^i d\theta^j
-2 P^0_{ik}\bar{Q}^{jk} d\theta^i dx_j
+ s^{-1}\bar{Q}^{ij} dx_i dx_j\\
&= s\bar{Q}_{ij} d\theta^i d\theta^j
+ \frac{\bar{Q}^{ij}}{s}\left( dx_i
- sP^0_{ik}d\theta^k\right) 
\left( dx_j
- sP^0_{jl}d\theta^l\right)\\
&= s\bar{Q}_{ij} d\theta^i d\theta^j\\
&\quad\quad+ \frac{\bar{Q}^{ij}}{s}\left\{ 
d\left( 
x_i-s\frac{\del \mathcal{H}}{\del\theta^i}\right)
- s\bar{P}_{ik}d\theta^k\right\} 
\left\{ d\left( x_j-s\frac{\del \mathcal{H}}{\del\theta^j}\right)
- s\bar{P}_{jl}d\theta^l\right\}.
\end{align*}
Now, define 
$F_s\colon \R^n\times T^n\to \R^n\times T^n$ by 
\begin{align*}
F_s(x,\theta):=
\left( x_1+s\frac{\del \mathcal{H}}{\del\theta^1},
\ldots, x_n+s\frac{\del \mathcal{H}}{\del\theta^n},
\theta\right).
\end{align*}
Then $F_{-s}$ is the inverse of 
$F_s$ and
\begin{align*}
F_s^*g_{A'}
= s\bar{Q}_{ij} d\theta^i d\theta^j
+ \frac{\bar{Q}^{ij}}{s}
\left( d x_i- s\bar{P}_{ik}d\theta^k\right)
\left( d x_j- s\bar{P}_{jl}d\theta^l\right)
\end{align*}
holds. 
Moreover, we can lift $F_s$ to 
\begin{align*}
\hat{F}_s\colon \R^n\times 
(\R^n/{\rm Ker}\Phi)\times S^1 \to 
\R^n\times 
(\R^n/{\rm Ker}\Phi)\times S^1
\end{align*}
by 
\begin{align*}
\hat{F}_s(x,\theta,e^{\sqrt{-1}t}):=
\left( x_1+s\frac{\del \mathcal{H}}{\del\theta^1},
\ldots, x_n+s\frac{\del \mathcal{H}}{\del\theta^n},
\theta,e^{\sqrt{-1}(t+s\mathcal{H}(\theta))}\right).
\end{align*}
One can easy to check that 
$\hat{F}_s$ is $\Z/m\Z$-equivariant 
and $S^1$-equivariant map, and 
\begin{align*}
\hat{F}_s^*\hat{g}_{A'}
&= \sigma(dt - x_id\theta^i)^2
+s\bar{Q}_{ij} d\theta^i d\theta^j\\
&\quad\quad + \frac{\bar{Q}^{ij}}{s}
\left( d x_i- s\bar{P}_{ik}d\theta^k\right)
\left( d x_j- s\bar{P}_{jl}d\theta^l\right).
\end{align*}
Put $\bar{P}=(\bar{P}_{ij})_{i,j}$, $\bar{Q}=(\bar{Q}_{ij})_{i,j}$, 
$\bar{\Theta}=\bar{Q}
+\bar{P}\bar{Q}^{-1}\bar{P}$ and 
$y=\sqrt{s\bar{\Theta}}^{-1}x$, $\tau=
\sqrt{s\bar{\Theta}}\theta$. 
Then we may write 
\begin{align*}
\hat{F}_s^*\hat{g}_{A'}
&= \sigma(dt - {}^tx \cdot d\theta)^2
+{}^t d\theta \cdot s\bar{Q} \cdot d\theta\\
&\quad\quad + 
{}^t\left( d x- s\bar{P}d\theta\right)
\cdot\frac{\bar{Q}^{-1}}{s} \cdot
\left( d x- s\bar{P}d\theta\right)\\
&= \sigma(dt)^2 - 2\sigma ({}^ty \cdot d\tau) dt 
+ {}^t d\tau \cdot \left( 1
+\sigma y\cdot{}^t y \right) \cdot d\tau\\
&\quad\quad + 
{}^t d y 
\cdot \sqrt{\bar{\Theta}}\bar{Q}^{-1}\sqrt{\bar{\Theta}} 
\cdot d y
-2\cdot{}^t d y 
\cdot \sqrt{\bar{\Theta}}\bar{Q}^{-1} \bar{P}\sqrt{\bar{\Theta}}^{-1}
\cdot d\tau. 
\end{align*}
Since $K_y:=1 +\sigma y\cdot{}^t y$ 
is positive definite, 
it has the inverse and the square root. 
Accordingly, we have 
\begin{align*}
\hat{F}_s^*\hat{g}_{A'}
&= \sigma(dt)^2 - 2\sigma ({}^ty \cdot d\tau) dt 
+ {}^t d\tau \cdot K_y \cdot d\tau\\
&\quad\quad + 
{}^t d y 
\cdot \sqrt{\bar{\Theta}}\bar{Q}^{-1}\sqrt{\bar{\Theta}} 
\cdot d y
-2\cdot{}^t d y 
\cdot \sqrt{\bar{\Theta}}\bar{Q}^{-1} \bar{P}\sqrt{\bar{\Theta}}^{-1}
\cdot d\tau\\
&= {}^t
\left( \sqrt{K_y}d\tau-\sigma \sqrt{K_y}^{-1}y dt 
- \sqrt{K_y}^{-1}\sqrt{\bar{\Theta}}^{-1}\bar{P}\bar{Q}^{-1}\sqrt{\bar{\Theta}}dy\right)\\
&\quad\quad \cdot 
\left( \sqrt{K_y}d\tau-\sigma \sqrt{K_y}^{-1}y dt 
- \sqrt{K_y}^{-1}\sqrt{\bar{\Theta}}^{-1}\bar{P}\bar{Q}^{-1}\sqrt{\bar{\Theta}}dy\right)\\
&\quad\quad+ 
\left( \sigma 
-(\sigma^2) {}^ty K_y^{-1}y\right)(dt)^2\\
&\quad\quad 
-2\sigma {}^ty \cdot K_y^{-1} \sqrt{\bar{\Theta}}^{-1}
\bar{P} \bar{Q}^{-1}
\sqrt{\bar{\Theta}} dy
dt\\
&\quad\quad+ {}^tdy\cdot \sqrt{\bar{\Theta}}
\left( 
\bar{Q}^{-1}
- \bar{Q}^{-1}\bar{P}
\sqrt{\bar{\Theta}}^{-1}K_y^{-1}
\sqrt{\bar{\Theta}}^{-1}\bar{P}\bar{Q}^{-1}\right) 
\sqrt{\bar{\Theta}}\cdot dy.
\end{align*}

Here, we have 
\begin{align*}
{}^ty \cdot K_y^{-1}
&= \frac{{}^ty}{1+\sigma \| y\|^2},\\
{}^ty \cdot K_y^{-1}\cdot y
&= \frac{\| y\|^2}{1+\sigma \| y\|^2},
\end{align*}
and the similar computation as in 
the proof of Proposition \ref{estimate on distance} 
gives 
\begin{align*}
\bar{\Theta}^{-1}=\bar{Q}^{-1}
-\bar{Q}^{-1}\bar{P}\bar{\Theta}^{-1}
\bar{P} \bar{Q}^{-1}.
\end{align*}
Put 
\begin{align*}
\mathcal{T}
&:= d\tau-\sigma K_y^{-1}\cdot y dt 
- K_y^{-1}\sqrt{\bar{\Theta}}^{-1}\bar{P}\bar{Q}^{-1}\sqrt{\bar{\Theta}}dy,\\
\bar{S}&:=\sqrt{\bar{\Theta}}^{-1}\bar{P}\bar{Q}^{-1}\sqrt{\bar{\Theta}}.
\end{align*}
Then we may write
\begin{align*}
\hat{F}_s^*\hat{g}_{A'}
&= {}^t \mathcal{T}\cdot K_y \cdot\mathcal{T}
+ \frac{\sigma}{1+\sigma\| y\|^2}(dt)^2
-\frac{2\sigma}{1+\sigma\| y\|^2} {}^ty \cdot 
\bar{S} dy dt\\
&\quad\quad+ {}^tdy\cdot \left( 
1
+ {}^t\bar{S}\left( 
1-K_y^{-1}
\right)\bar{S}\right) \cdot dy.
\end{align*}
Since we have 
\begin{align*}
1-K_y^{-1}
= (K_y-1)K_y^{-1}
=\sigma y\cdot{}^ty\cdot K_y^{-1}
=\frac{\sigma y\cdot {}^ty}{1+\sigma\| y\|^2},
\end{align*}
we can see that 
\begin{align*}
\hat{F}_s^*\hat{g}_{A'}
&= {}^t \mathcal{T}\cdot K_y \cdot\mathcal{T}
+ 
\frac{\sigma}{1+\sigma\| y\|^2}(dt)^2
-\frac{2\sigma}{1+\sigma\| y\|^2} {}^ty \cdot \bar{S} dy dt\\
&\quad\quad+ {}^tdy\cdot \left( 
1
+ {}^t\bar{S}\left( 
\frac{\sigma y\cdot {}^ty}{1+\sigma\| y\|^2}
\right)\bar{S}\right) \cdot dy\\
&= {}^t \mathcal{T}\cdot K_y \cdot\mathcal{T}
+ 
\frac{\sigma}{1+\sigma\| y\|^2}
\left( dt-{}^ty \cdot \bar{S} dy\right)^2
+{}^tdy\cdot dy
\end{align*}
Define $\phi_{m,s}\colon S_{\Phi,R} \to \R^n\times S^1$ 
by 
\begin{align*}
\phi_{m,s}(x,\theta,e^{\sqrt{-1}t})
:= (\sqrt{s\bar{\Theta}}^{-1}x,e^{\sqrt{-1}t}). 
\end{align*}
and define $\Z/m\Z$-action on $\R^n\times S^1$ 
by $k\cdot (y,e^{\sqrt{-1}t}):=(y,e^{\sqrt{-1}(t-\frac{2k\pi}{m})})$. 
Then $\phi_{m,s}$ is $\Z/m\Z$-equivariant map 
and 
\begin{align*}
\phi_{m,s}\colon 
(\R^n\times (\R^n/{\rm Ker}\,\Phi)
\times S^1, \hat{F}_s^*\hat{g}_{A'})
\to (\R^n\times S^1,g_\infty)
\end{align*}
is a Riemannian submersion, where 
\begin{align*}
g_\infty
= \frac{\sigma}{1+\sigma\| y\|^2}
\left( dt-{}^ty \cdot \bar{S} dy\right)^2
+{}^tdy\cdot dy.
\end{align*}
Denote by $\mu_\infty$ the 
measure on $\R^n\times S^1$ defined by 
$d\mu_\infty = dy_1\cdots dy_n dt$.

\begin{prop}
Let $f\in C_0(\R^n\times S^1)$. 
Then there is a constant $K>0$ depending only on 
$\Phi,\sigma,\bar{\Theta}$ such that 
\begin{align*}
\int_{\R^n\times (\R^n/{\rm Ker}\,\Phi)\times S^1}
f\circ \phi_{m,s} d\mu_{\hat{F}_s^*\hat{g}_{A'}} 
= K\sqrt{s}^n \int_{\R^n\times S^1}
f d\mu_\infty.
\end{align*}
\label{push of measure}
\end{prop}
\begin{proof}
Since 
\begin{align*}
d\mu_{\hat{F}_s^*\hat{g}_{A'}} &= 
\left( \frac{\sigma}{1+\sigma\| y\|^2}\det (K_y)
\right)^\frac{1}{2}dt d\tau^1\cdots d\tau^n 
dy_1\cdots dy_n\\
&= 
\left( \frac{\sigma}{1+\sigma\| y\|^2}(1+\sigma\| y\|^2)
\right)^\frac{1}{2}dt d\tau^1\cdots d\tau^n 
dy_1\cdots dy_n\\
&= \sqrt{\sigma}dt d\tau^1\cdots d\tau^n 
dy_1\cdots dy_n,
\end{align*}
we have 
\begin{align*}
&\quad\ \int_{\R^n\times (\R^n/{\rm Ker}\,\Phi)\times S^1}
f\circ \phi_{m,s} d\mu_{\hat{F}_s^*\hat{g}_{A'}} \\
&= \int_{\R^n\times (\R^n/{\rm Ker}\,\Phi)\times S^1} 
f\circ \phi_{m,s} 
\sqrt{\sigma}dt d\tau^1\cdots d\tau^n 
dy_1\cdots dy_n\\
&= \sqrt{\sigma}
\sqrt{s}^n\sqrt{\det{\bar{\Theta}}}
\int_{\R^n\times (\R^n/{\rm Ker}\,\Phi)\times S^1} 
f\circ \phi_{m,s} 
dt d\theta^1\cdots d\theta^n 
dy_1\cdots dy_n\\
&= \sqrt{\sigma}
{\rm Vol}\left( \R^n/{\rm Ker}\,\Phi\right)
\sqrt{\det{\bar{\Theta}}}\sqrt{s}^n
\int_{\R^n\times S^1} fdt dy_1\cdots dy_n.
\end{align*}
\end{proof}
Now, we put 
\begin{align*}
\mathbf{S}_\Phi:=\frac{\R^n\times (\R^n/{\rm Ker}\,\Phi)
\times S^1}{\Z/m\Z},
\end{align*}
then $\hat{g}_{A'}$ and 
$\hat{F}_s^*\hat{g}_{A'}$ induces the 
Riemannian metrics on $\mathbf{S}_\Phi$ 
such that $p_\Phi$ is local isometry. 
We also denote by $\hat{g}_{A'}$ and 
$\hat{F}_s^*\hat{g}_{A'}$, respectively 
if there is no fear of confusion.

Since $\phi_{m,s}$ is $\Z/m\Z$-equivariant, 
we have the following commutative 
diagram; 
\[ 
\left.
\begin{array}{ccc}
(\R^n\times (\R^n/{\rm Ker}\,\Phi)
\times S^1, \hat{F}_s^*\hat{g}_{A'}) & \stackrel{\phi_{m,s}}{\rightarrow} & (\R^n\times S^1,g_\infty) \\
p_\Phi \downarrow & & p_m \downarrow \\
(\mathbf{S}_\Phi, \hat{F}_s^*\hat{g}_{A'}) & \stackrel{\phi_s}{\rightarrow} & 
(\R^n\times S^1,g_{m,\infty})
\end{array}
\right.
\]
where $p_m$ is the quotient map defined 
by $p_m(y,e^{\sqrt{-1}t}):=(y,e^{\sqrt{-1}mt})$ 
and $g_{m,\infty}$ is defined by 
\begin{align}
g_{m,\infty}
= \frac{\sigma}{1+\sigma\| y\|^2}
\left( \frac{dt}{m}-{}^ty \cdot \bar{S} dy\right)^2
+{}^tdy\cdot dy
\label{limit metric}
\end{align}
such that 
${p_m}^*g_{m,\infty}=g_\infty$ and 
$\phi_s$ is the Riemannian submersion.

\begin{prop}
Let $\{ J_s\}_s$ satisfy $\spadesuit$ 
and $A'(s,x,\theta):=sA^0(0,\theta)$ and put 
$p_0=p_\Phi(0,0,1)\in \mathbf{S}_\Phi$. 
Assume that there are constants $s_0>0$ and 
$\kappa\in\R$ 
such that ${\rm Ric}_{g_{J_s}}\ge \kappa g_{J_s}$ for any $0<s\le s_0$. 
Then the family of pointed metric measure spaces with 
the isometric $S^1$-action 
\begin{align*}
\left\{ \left( \mathbf{S}_\Phi, d_{A'}, \frac{\mu_{\hat{g}_{A'}}}{K\sqrt{s}^n},p_0\right)\right\}_s
\end{align*}
converges to $\left( \R^n\times S^1, d_{g_{m,\infty}}, \mu_{\infty},(0,1)\right)$ as $s\to 0$ 
in the sense of the pointed $S^1$-equivariant 
measured Gromov-Hausdorff topology. 
\label{conv of local model}
\end{prop}
\begin{proof}
Since $\hat{F}_s$ is an $S^1$-equivariant 
isometry, 
it suffices to show that 
\begin{align*}
\left\{ \left( \mathbf{S}_\Phi, d_{\hat{F}_s^*\hat{g}_{A'}}, \frac{\mu_{\hat{F}_s^*\hat{g}_{A'}}}{K\sqrt{s}^n},p_0\right)\right\}_s
\end{align*}
converges to $\left( \R^n\times S^1, d_{g_{m,\infty}}, \mu_{\infty},(0,1)\right)$ as $s\to 0$ 
in the sense of the pointed $S^1$-equivariant 
measured Gromov-Hausdorff topology. 
Since 
\begin{align*}
\hat{F}_s^*\hat{g}_{A'}
= {}^t\mathcal{T}K_y\mathcal{T} + g_\infty,
\end{align*}
one can see that 
$\phi_s$ is a Riemannian submersion and 
the diameters of the fibers $\phi_s^{-1}(y,t)$ 
are at most $C\sqrt{s(1+\sigma\| y\|^2)}
$, where $C>0$ is a constant 
depending only on $\bar{P},\bar{Q}$ and $\Phi$, 
hence the 
pointed Gromov-Hausdorff convergence follows. 
Moreover, Proposition \ref{push of measure} 
implies that 
$(\phi_{m,s})_*\mu_{\hat{F}_s^*hat{g}_{A'}}
=K\sqrt{s}^n\mu_\infty$, 
especially we also have the vague 
convergence of the measures. 
\end{proof}
\begin{thm}
Let $\{ J_s\}_s$ satisfies $\spadesuit$ and 
suppose that there there are constants 
$s_0>0$ and $\kappa\in\R$ 
such that ${\rm Ric}_{g_{J_s}}\ge \kappa g_{J_s}$ 
for any $0<s\le s_0$. 
Put 
$p_0=p_\Phi(0,0,1)\in S(L|_U,h)$. 
Then the family of pointed metric measure spaces with 
the isometric $S^1$-action 
\begin{align*}
\left\{ \left( S(L,h), d_{J_s}, 
\frac{\mu_{\hat{g}_{J_s}}}{K\sqrt{s}^n},p_0
\right)\right\}_s
\end{align*}
converges to $\left( \R^n\times S^1, d_{g_{m,\infty}}, \mu_{\infty},(0,1)\right)$ as $s\to 0$ 
in the sense of the pointed $S^1$-equivariant 
measured Gromov-Hausdorff topology. 
\label{conv 1}
\end{thm}
\begin{proof}
Put $A'(s,x,\theta):=sA^0(0,\theta)$. 
By Proposition \ref{approximation}, 
there exist constants $R_0,C>0$ such that 
\begin{align*}
{\rm id}\colon \left( 
\pi^{-1}(B_{g_{J_s}}(\mathbf{0},r-C\sqrt{s}r^2)),d_{J_s}\right) \to 
\left( \pi^{-1}(B_{g_{A'(s,\cdot)}}(\mathbf{0},r)), d_{A'(s,\cdot)}\right)
\end{align*}
is a Borel 
$C\sqrt{s}r^2$-$S^1$-equivariant 
Hausdorff approximation 
for any $r\ge 1$ and $s\le \frac{R_0^2}{Cr^2}$. 
Since $C\sqrt{s}r^2\to 0$ as $s\to 0$ for 
any fixed $r$, therefore, 
\begin{align*}
\left\{ \left( S(L,h), d_{J_s}, p_0
\right)\right\}_s
\stackrel{S^1\mathchar`-{\rm GH}}{\longrightarrow} \left( \R^n\times S^1, d_{g_{m,\infty}}, (0,1)\right)
\end{align*}
as $s\to 0$ by Proposition 
\ref{conv of local model}. 

Next we show the vague convergence 
of the measures. 
Now the approximation from 
$\left( S(L,h), d_{J_s}, p_0\right)$ 
to $\left( \R^n\times S^1, d_{g_{m,\infty}}, (0,1)\right)$ is induced by the 
$\Z/m\Z$-equivariant maps 
$\psi_s:=\phi_{m,s}\circ\hat{F}_{-s}$. 
Take $f\in C_0(\R^n\times S^1)$. 
Then Proposition \ref{push of measure} 
gives 
\begin{align*}
\int_{\R^n\times S^1}
f d\mu_\infty
&= \frac{1}{K\sqrt{s}^n}
\int_{\R^n\times (\R^n/{\rm Ker}\,\Phi)\times S^1}
f\circ \phi_{m,s} d\mu_{\hat{F}_s^*\hat{g}_{A'}}\\
&= \frac{1}{K\sqrt{s}^n}
\int_{\R^n\times (\R^n/{\rm Ker}\,\Phi)\times S^1}
f\circ \psi_s d\mu_{\hat{g}_{A'}}.
\end{align*}
Note that $\sup |f\circ \psi_s| \le \sup |f|<\infty$. 
By the definition of $\phi_{m,s}$, 
there is $r>0$ independent of $s$ such that 
${\rm supp}(f\circ \psi_s) \subset S_{\sqrt{s}r}$ 
holds for any $0<s\le s_0$. 
Then Proposition \ref{approximation} gives 
some constants $C_2>0$ such that 
\begin{align*}
&\quad\quad 
\frac{1}{K\sqrt{s}^n}
\left| 
\int_{\R^n\times (\R^n/{\rm Ker}\,\Phi)\times S^1}
f\circ \psi_s d\mu_{\hat{g}_{J_s}}
- \int_{\R^n\times (\R^n/{\rm Ker}\,\Phi)\times S^1}
f\circ \psi_s d\mu_{\hat{g}_{A'}}
\right|\\
&\le \frac{C_2\sup|f|(\sqrt{s}r)^{n+1}}{K\sqrt{s}^n}
\to 0
\end{align*}
as $s\to 0$. 
\end{proof}

\section{The spectral structures on the limit spaces}\label{sec spectral}
In this section we consider the 
metric measure space $(\R^n\times S^1,g_{m,\infty},\mu_\infty)$
defined by \eqref{limit metric}. 
Now, note that 
\begin{align*}
\bar{S}=\sqrt{\bar{\Theta}}^{-1}\bar{P} \bar{Q}^{-1}
\sqrt{\bar{\Theta}}
&=\sqrt{\bar{\Theta}}^{-1}\left( 
\bar{P} \bar{Q}^{-1}\bar{\Theta}
\right)
\sqrt{\bar{\Theta}}^{-1}\\
&=\sqrt{\bar{\Theta}}^{-1}\left( 
\bar{P} + \bar{P} \bar{Q}^{-1}\bar{P} \bar{Q}^{-1}\bar{P}\right)
\sqrt{\bar{\Theta}}^{-1},
\end{align*}
which implies $\bar{S}$ 
is symmetric. 
Consequently, we can see 
\begin{align*}
\frac{dt}{m}-{}^ty \cdot \bar{S} \cdot dy
&= d\left( \frac{t}{m}-\frac{{}^ty \cdot 
\bar{S} 
\cdot y}{2}
\right).
\end{align*}
Here, by taking the pullback of $g_{m,\infty}$ 
by the diffeomorphism
\[ 
\left.
\begin{array}{ccc}
\R^n\times S^1 & \rightarrow & \R^n\times S^1 \\
\rotatebox{90}{$\in$} & & \rotatebox{90}{$\in$} \\
\left( y,e^{\sqrt{-1}t}\right) & \mapsto & \left( y,e^{\sqrt{-1}\left( t+m\cdot
\frac{{}^ty \cdot 
\bar{S}
\cdot y}{2} \right)}\right),
\end{array}
\right.
\]
we may suppose 
\begin{align*}
g_{m,\infty}
&= \frac{\sigma}{m^2(1+\sigma\| y\|^2)}
( dt )^2
+{}^tdy\cdot dy\\
d\mu_\infty &= dy_1\cdots dy_ndt
\end{align*}
and the isometric 
$S^1$-action on $(\R^n\times S^1,g_{m,\infty})$ 
is given by 
\begin{align*}
e^{\sqrt{-1}\tau}\cdot
\left( y,e^{\sqrt{-1}t}\right)
=\cdot
\left( y,e^{\sqrt{-1}(t+m\tau)}\right). 
\end{align*}
Then the Laplace operator $\Delta_{m,\infty}$ 
on $(\R^n\times S^1,g_{m,\infty},\mu_\infty)$
is defined such that 
\begin{align*}
\int_{\R^n\times S^1}
(\Delta_{m,\infty} f_1) f_2 d\mu_\infty
= \int_{\R^n\times S^1}\langle 
df_1,df_2\rangle_{g_{m,\infty}} d\mu_\infty
\end{align*}
holds for any $f_1,f_2\in C^\infty(\R^n\times S^1)$, 
therefore we have 
\begin{align*}
\Delta_{m,\infty} f
= \Delta_{\R^n} f -\frac{m^2(1+\sigma \| y\|^2)}{\sigma} \frac{\del^2 f}{\del t^2},
\end{align*}
where $\Delta_{\R^n}=-\sum_{i=1}^n\frac{\del^2}{\del y_i^2}$. 

Let $\rho_k$ be the representation of $S^1$ 
defined in Section \ref{principal metric}, 
then we have 
\begin{align*}
\left( L^2 (\R^n\times S^1)
\otimes \C\right)^{\rho_{ml}}
= \left\{ \varphi(y)e^{-\sqrt{-1}l t};\, \varphi\in L^2 (\R^n)\right\}
\end{align*}
and $\left( L^2 (\R^n\times S^1)
\otimes \C\right)^{\rho_{k}}=\{ 0\}$ if $k\notin m\Z$. 
Now we consider the operator
\begin{align*}
\Delta_{m,\infty} - \left( \frac{k^2}{\sigma} + kn\right)\colon 
\left( C^\infty (\R^n\times S^1)
\otimes \C\right)^{\rho_k}
\to 
\left( C^\infty (\R^n\times S^1)
\otimes \C\right)^{\rho_k}
\end{align*}
for $k=ml$, 
which corresponds to 
the limit of 
\begin{align*}
2\Delta_{\delb_{J_s}}\colon 
C^\infty (X,L^k) \to C^\infty (X,L^k)
\end{align*}
as $s\to 0$. 
Let $(\R^n,{}^tdy\cdot dy,
e^{-k\| y\|^2}d\mathcal{L}_{\R^n})$ be the 
Gaussian space, 
where $\mathcal{L}_{\R^n}$ is the Lebesgue 
measure on $\R^n$ and denote by 
$\Delta_{\R^n,k}$ 
the Laplacian 
of this metric measure space. 
Note that we have 
\begin{align*}
\Delta_{\R^n,k} \varphi
= \Delta_{\R^n}\varphi +2k\sum_{i=1}^n y_i\frac{\del \varphi}{\del y_i}.
\end{align*}
Then we can see that 
the following linear isomorphism 
\[ 
\left.
\begin{array}{ccc}
C^\infty (\R^n)\otimes\C & \to 
& \left( C^\infty (\R^n\times S^1)
\otimes \C\right)^{\rho_k} \\
\rotatebox{90}{$\in$} & & \rotatebox{90}{$\in$} \\
\varphi & \mapsto & 
\varphi\cdot e^{-\frac{k\| y\|^2+\sqrt{-1}lt}{2}}
\end{array}
\right.
\]
induces the isomorphism 
\begin{align*}
L^2(\R^n,e^{-k\| y\|^2}d\mathcal{L}_{\R^n})
\otimes\C
\cong 
\left( L^2 (\R^n\times S^1,d\mu_\infty)
\otimes \C\right)^{\rho_{k}}
\end{align*}
and the identification of the operators 
\begin{align*}
\Delta_{\R^n,k}
\cong 
\Delta_{m,\infty} - \left( \frac{k^2}{\sigma} + kn\right).
\end{align*}
Next we construct 
the eigenfunctions of $\Delta_{\R^n,k}$
by the hermitian polynomials. 
For $\xi\in\R$ the hermitian polynomials are 
defined by 
\begin{align*}
H_{k,N}(\xi) := e^{k\xi^2}
\frac{d^N}{d\xi^N} e^{-k\xi^2}, 
\end{align*}
which is a polynomial in $\xi$ of degree $N$, 
then it is known that $H_{k,N}$ solves 
\begin{align*}
-\frac{d^2}{d\xi^2}H_{k,N} 
+2k\xi\frac{d}{d\xi}H_{k,N} 
= 2kN H_{k,N}
\end{align*}
and 
$\{ H_{k,N}\}_{N=0}^\infty$ is 
a complete orthonormal system 
of $L^2(\R,e^{-k\xi^2}d\mathcal{L}_\R)$. 
Let $N=(N_1,\ldots,N_n)\in\Z_{\ge 0}^n$ and 
put 
\begin{align*}
\left(\frac{\del}{\del y}\right)^N
&:= \frac{\del^{N_1}}{\del y_1^{N_1}}
\cdots \frac{\del^{N_n}}{\del y_n^{N_n}},\\
|N|&:=\sum_{i=1}^n N_i. 
\end{align*}
Then 
\begin{align*}
\varphi(y)= \prod_{i=1}^n H_{k,N_i}(y_i)
= e^{k\| y\|^2}\left(\frac{\del}{\del y}\right)^N
(e^{-k\| y\|^2})
\end{align*}
solves 
\begin{align*}
\Delta_{\R^n,k} \varphi 
= 2k|N| \varphi
\end{align*}
and $\{ \prod_{i=1}^n H_{k,N_i}(y_i);\, (N_1,\ldots,N_n)\in\Z_{\ge 0}\}$ is a complete orthonormal system 
of $L^2(\R^n,e^{-k\| y\|^2}d\mathcal{L}_{\R^n})$. 
Thus we have the following theorem.

\begin{thm}
Let $l\in\Z_{>0}$, $k=ml$ and 
\begin{align*}
W(k,\lambda) := \left\{ 
f\in \left( C^\infty (\R^n\times S^1)
\otimes \C\right)^{\rho_k};\,
\left( \Delta_{m,\infty} - \frac{k^2}{\sigma} - kn\right)f = 2\lambda f \right\}.
\end{align*}
Then there is an orthogonal decomposition 
\begin{align*}
(L^2(\R^n\times S^1)\otimes \C)^{\rho_k}
&=\overline{ \bigoplus_{d\in\Z_{\ge 0}}W(k,kd)},
\end{align*}
where 
\begin{align*}
W(k,kd)
&= {\rm span}_\C\left\{ e^{\frac{k\| y\|^2}{2}-\sqrt{-1}lt}
\left(\frac{\del}{\del y}\right)^N
(e^{-k\| y\|^2});\, 
N\in\Z_{\ge 0},\, 
|N|=d\right\}.
\end{align*}
\label{thm spec}
\end{thm}
As a consequence of 
Theorem \ref{thm spec}, 
we obtain the former part 
of Theorem \ref{main 3}.

\section{The fibers which are not 
$m$-BS fibers for any positive $m$}\label{sec others}
In this section we suppose 
$(X^{2n},\omega)$ is a symplectic manifold 
with a prequantum line bundle $(L,\nabla,h)$, 
and assume that there is 
a continuous map $\mu\colon X \to Y$ to a 
topological space $Y$. 
Moreover we fix $b_0\in Y$ such that 
$\mu^{-1}(b_0)$ is not an $m$-BS fiber 
for any $m\in\Z$. 

Let $\{ J_s\}_{0<s\le s_0}$ be a 
one parameter family of $\omega$-compatible 
complex structures, and 
denote by $\mathcal{L}_{g}(c)$ the 
length of a path $c$ 
with respect to the Riemannian metric $g$. 
We fix $p_0\in\mu^{-1}(b_0)$ 
and assume the 
followings. 

\begin{itemize}
\setlength{\parskip}{0cm}
\setlength{\itemsep}{0cm}
 \item[$\star 1$] 
For any $r>0$ and 
open neighborhood $B\subset Y$ of 
$b_0$ there is $s_{r,B}>0$ 
such that 
\begin{align*}
\mu(B_{g_{J_s}}(p_0,r)) 
\subset B
\end{align*}
holds for any $s\le s_{r,B}$. 
 \item[$\star 2$] 
For any piecewise smooth closed path 
$c_{b_0}\colon [0,1] \to X$ 
there exist an open neighborhood 
$B$ of $b_0$ and 
a continuous map $c\colon B\times [0,1] \to X$ such that 
$\mu\circ c(b,t) = b$, $c(b,0)=c(b,1)$, 
$c(b_0,\cdot )=c_{b_0}$ and $c(b,\cdot)$ are piecewise smooth. 
 \item[$\star 3$] 
For any open neighborhood 
$B$ of $b_0$ and a 
continuous map 
$c\colon B\times [0,1] \to X$ such that 
$\mu\circ c(b,t) = b$ and $c(b,\cdot)$ are  piecewise smooth, 
\begin{align*}
\lim_{s\to 0}\sup_{b\in B}\mathcal{L}_{g_{J_s}}(c(b,\cdot))
=0
\end{align*}
holds.
\end{itemize}

Let $\pi\colon S(L,h) \to X$ be 
the natural projection. 
By the connection $\nabla$ we have 
the unique horizontal lift 
$\tilde{c}\colon [0,1]\to S(L,h)$ 
with $\tilde{c}(0)=u_0$ 
for any pair of a piecewise smooth 
path $c\colon [0,1]\to X$ and 
$u_0\in \pi^{-1}(c(0))$. 
\begin{prop}
Assume that 
$\mu^{-1}(b_0)$ is not an $m$-BS fiber 
for any $m\in\Z$. 
For any $p_0\in\mu^{-1}(b_0)$, 
$e^{\sqrt{-1}t}\in S^1$ 
and $\delta>0$, 
there is a piecewise smooth path 
$c\colon [0,1]\to \mu^{-1}(b_0)$ 
with $c(0)=c(1)=p_0$ such that 
its horizontal lift $\tilde{c}$ 
satisfying 
$\tilde{c}(1)=\tilde{c}(0)e^{\sqrt{-1}t'}$ 
and $|t'-t|<\delta$. 
In particular, if we assume $\star 3$, 
then 
$\lim_{s\to 0}{\rm diam}_{\hat{g}_{J_s}}( \pi^{-1}(p_0))
=0$ holds. 
\label{shrink 1}
\end{prop}
\begin{proof}
Since $\mu^{-1}(b_0)$ is not an $m$-BS fiber 
for any $m\in\Z$, 
the holonomy group of $\nabla|_{\mu^{-1}(b_0)}$ 
may not contained in any proper closed subgroup 
of $S^1$, hence we obtain 
the path $c$ which satisfies the assertion. 
By $\star 3$, 
\begin{align*}
\lim_{s\to 0}\mathcal{L}_{\hat{g}_{J_s}}(\tilde{c})
= \lim_{s\to 0}\mathcal{L}_{g_{J_s}}(c)
= 0
\end{align*}
holds, hence $d_{J_s}(\tilde{c}(0),\tilde{c}(1))
\to 0$ as $s\to 0$. 
Therefore, for any $u_0\in\pi^{-1}(p_0)$, 
$e^{\sqrt{-1}t}\in S^1$ and $\delta$ 
we have 
\begin{align*}
\lim_{s\to 0}d_{J_s}(u_0,u_0e^{\sqrt{-1}t})
\le 
\lim_{s\to 0}d_{J_s}(u_0,u_0e^{\sqrt{-1}t'}) 
+\sigma|t-t'|
<\sigma\delta, 
\end{align*}
which implies $\lim_{s\to 0}d_{J_s}(u_0,u_0e^{\sqrt{-1}t})=0$, 
hence we have 
\begin{align*}
\lim_{s\to 0}{\rm diam}_{\hat{g}_{J_s}}( \pi^{-1}(p_0))
=0.
\end{align*}
\end{proof}

Let $B\subset Y$ be open 
and $\tilde{c}\colon B\times [0,1]\to S(L,h)$ 
be a map such that 
$\tilde{c}_y:=\tilde{c}(y,\cdot)$ is 
one of 
the horizontal lift of $c_y:=c(y,\cdot)$ 
with respect to $\nabla$. 
Let $t_y\in\R$ be defined by 
$\tilde{c}(y,1)=\tilde{c}(y,0)e^{\sqrt{-1}t_y}$, 
which is determined independent of 
the choice of the initial point of $\tilde{c}(y,\cdot)$. 
Then the map $y\mapsto e^{\sqrt{-1}t_y}$ is 
continuous. 

For a sufficiently large integer $N>0$, 
put $t=\frac{2\pi}{N}$ and $\delta=t=\frac{\pi}{N}$ 
and take $c$ and $t'$ as in Proposition 
\ref{shrink 1}, 
then we extend 
$c$ to $c\colon B\times [0,1]\to X$ 
by $\star 2$, where $B$ is an open neighborhood 
of $b_0$. 
Then by the continuity of $e^{\sqrt{-1}t_y}$, 
there is an open neighborhood 
$B_N\subset Y$ of $b_0$ 
such that $\frac{\pi}{N} < t_y< \frac{3\pi}{N}$ 
holds for any $y\in B_N$. 
If we consider the path obtained by 
connecting $k$ copies of $c_y$, 
we can see that 
\begin{align*}
d_{J_s}(\tilde{c}_y(0),\tilde{c}_y(0)e^{\sqrt{-1}kt_y})
\le k\mathcal{L}_{g_{J_s}}(c_y).
\end{align*}
If we consider the path along the fiber of 
$S(L,h)\to X$, we have 
\begin{align*}
d_{J_s}(\tilde{c}_y(0)e^{\sqrt{-1}a},\tilde{c}_y(0)e^{\sqrt{-1}b})
\le \sigma|a-b|.
\end{align*}
Combining these estimates, we can see 
\begin{align*}
d_{J_s}(\tilde{c}_y(0),\tilde{c}_y(0)e^{\sqrt{-1}\theta})
\le N\mathcal{L}_{g_{J_s}}(c_y)
+ \frac{3\pi\sigma}{N}
\end{align*}
for any $\theta\in\R$, which gives 
\begin{align*}
{\rm diam}_{\hat{g}_{J_s}}(\pi^{-1}(\tilde{c}_y(0)))\le N\mathcal{L}_{g_{J_s}}(c_y)
+ \frac{3\pi\sigma}{N}.
\end{align*}
Now we can take $s_N>0$ by 
$\star 3$ such that 
$\mathcal{L}_{g_{J_s}}(c_y)\le \frac{1}{N^2}$ 
for any $0<s\le s_N$ and $y\in B_N$. 
We can also take $0<s_{N,r}\le s_N$ by 
$\star 1$ such that 
$\mu(B_{g_{J_s}}(p_0,r))\subset B_N$ 
holds for all $0<s \le s_{N,r}$. 
Then  
we have 
\begin{align*}
{\rm diam}_{\hat{g}_{J_s}}(\pi^{-1}(\tilde{c}_y(0)))
\le \frac{1+3\pi\sigma}{N}
\end{align*}
for all $y\in \mu(B_{g_{J_s}}(p_0,r))$ 
and $0<s\le s_{N,r}$. 
Thus we obtain the following proposition. 
\begin{prop}
Assume $\star 1\mathchar`- 3$, 
$\mu^{-1}(b_0)$ is not an $m$-BS fiber for any $m$ 
and let $u_0\in\pi^{-1}(p_0)$. 
Then for any $r>0$ and 
$\varepsilon>0$ there is $0<s_{r,\varepsilon}\le s_0$ 
such that 
\begin{align*}
{\rm diam}_{\hat{g}_{J_s}}(\pi^{-1}(x))
\le \varepsilon
\end{align*}
for all $x\in B_{g_{J_s}}(p_0,r)$ and $0<s\le s_{r,\varepsilon}$. 
\end{prop}

Before we prove Theorem \ref{main 2}, 
we describe the relation between the convergence 
of principal $G$-bundles and the 
convergence of the base spaces. 
Let $G$ be a compact Lie group, 
$(P,d,\nu)$ be a metric measure space 
with an isometric $G$-action. 
Put $X:=P/G$ and define the distance $\bar{d}$ 
on $X$ by 
\begin{align*}
\bar{d}(\bar{x},\bar{y}):=\inf_{\gamma\in G}d(x,y\gamma),
\end{align*}
where $\bar{x}\in X$ is the equivalence class represented by $x\in P$. 

\begin{prop}
Let $\{ (P_i,d_i,\nu_i,p_i)\}_{i\in\N}$ be a 
sequence of pointed
metric measure spaces with isometric $G$ actions 
and denote by $\pi_i\colon P_i\to X_i=P_i/G$ 
be the quotient maps. 
Suppose that for any $r,\varepsilon>0$ 
there is $i_{r,\varepsilon}\in\N$ 
such that 
\begin{align*}
\sup_{x\in B(p_i,r)}{\rm diam}_{d_i} \pi^{-1}(x) <\varepsilon
\end{align*}
holds for any $i\ge i_{r,\varepsilon}$. 
If 
$\{ (X_i,\bar{d}_i,\bar{\nu}_i,\bar{p}_i)\}_i$ 
converges to 
$(X,\bar{d},\bar{\nu},\bar{p})$ 
with respect to the pointed 
measured Gromov-Hausdorff topology, 
then $\{ (P_i,d_i,\nu_i,p_i)\}_s$ 
converges to $(X,\bar{d},\bar{\nu},\bar{p})$ 
in the sense of the pointed $G$-equivariant 
measured Gromov-Hausdorff topology. 
Here, the $G$-action on $X$ 
is the trivial action. 
\label{quotient conv}
\end{prop}
\begin{proof}
Let $\bar{\phi}_i\colon (B_{X_i}(\bar{p}_i,r),\bar{p}_i) \to (X,\bar{p})$ be 
$\varepsilon$-approximations 
given by the pointed Gromov-Hausdorff 
convergence of $(X_i,\bar{p}_i)$.
Then one can see that 
$\phi:=\bar{\phi}_i\circ \pi_i\colon (\pi_i^{-1}(B(\bar{p}_i,r)),p_i) \to (X,\bar{p})$ 
are $G$-equivariant $2\varepsilon$-approximations. 
Using these maps 
one can show the assertion.  
\end{proof}

\begin{proof}[Proof of Theorem \ref{main 2}]
Assume $\spadesuit$ and 
that there is $\kappa\in \R$ such that 
${\rm Ric}_{g_{J_s}} \ge \kappa g_{J_s}$. 
Let $u\in S|_{\mu^{-1}(y)}$ and assume 
that $\mu^{-1}(y)$ is not a Bohr-Sommerfeld 
fiber of $L^m$ for any $m>0$. 
On the neighborhood $U$ of $\mu^{-1}(y)$, 
we may write 
\begin{align*}
g_{J_s}|_U=g_A
\end{align*}
for some $A=A(s,x,\theta)$. 
Here we consider the pointed 
measured Gromov-Hausdorff 
limit of $(X,g_{J_s},\frac{\mu_{g_{J_s}}}{K\sqrt{s}^n},p)$ 
as $s\to 0$ 
for some $p\in \mu^{-1}(y)$ and 
$K>0$. 
In the same way as Subsection \ref{subsec conv}, 
it suffices to consider the limit of 
$g_{A'}$, where $A'(s,x,\theta) = sA^0(0,\theta)$ 
and 
$\bar{Q}={\rm Im}(A^0)(0,\theta) $ 
is independent of $\theta$. 
Notice that we already had 
$P_{ij}^0(0,\theta)=\bar{P}_{ij}+\frac{\del^2\mathcal{H}}{\del \theta^i\del\theta^j}$ 
in Subsection \ref{subsec conv} and 
\begin{align*}
F_s^*g_{A'} &=
{}^t\left( \sqrt{s\bar{\Theta}}d\theta - \sqrt{s\bar{\Theta}}^{-1}
\bar{P}\bar{Q}^{-1} dx\right)\cdot 
\left( \sqrt{s\bar{\Theta}}d\theta - \sqrt{s\bar{\Theta}}^{-1}
\bar{P}\bar{Q}^{-1} dx\right)\\
&\quad\quad
+ s^{-1}\cdot {}^t dx\cdot 
\bar{\Theta}^{-1}\cdot dx
\end{align*}
holds by \eqref{principal metric eq}, 
where 
\begin{align*}
F_s(x,\theta) &=
\left( x_1+s\frac{\del \mathcal{H}}{\del\theta^1},
\ldots, x_n+s\frac{\del \mathcal{H}}{\del\theta^n},
\theta\right),\\
\bar{\Theta}&=\bar{Q}+\bar{P}\bar{Q}^{-1}\bar{P}.
\end{align*}
Then by the transformation 
$y=\sqrt{s\bar{\Theta}}^{-1}x$ 
and $\tau=\sqrt{s\bar{\Theta}} \theta$, 
we have 
\begin{align*}
F_s^*g_{A'}
&= 
{}^t\left( d\tau -
\bar{P}\bar{Q}^{-1} dy\right)\cdot 
\left( d\tau -
\bar{P}\bar{Q}^{-1} dy\right)
+ {}^t dy \cdot dy.
\end{align*}
The above expression implies 
that $(y,\tau)\mapsto y$ is 
the Riemannian submersion to 
the Euclidean space. 
Since the diameters of 
the fibers of the submersion converge to $0$ 
as $s\to 0$, we have proved that 
$(X,F_s^*g_{A'},p_0)$ pointed Gromov-Hausdorff 
converges to $(\R^n,{}^tdy\cdot dy,0)$. 
The convergence of the measure 
is shown by the similar argument with 
the proof of 
Proposition \ref{push of measure} 
and Theorem \ref{conv 1}. 
By Proposition \ref{quotient conv} 
we obtain the assertion. 
\end{proof}

As a consequence of \ref{main 2}
we obtain the latter half 
of Theorem \ref{main 3}, 
since the $S^1$-action on 
$\R^n$ in Theorem \ref{main 2} 
is trivial and $(C^{\infty}(\R^n)\otimes\C)^{\rho_k}
=\{ 0\}$ for any $k>0$.

\paragraph{\bf Acknowledgment.}
The author would like to express his 
gratitude to 
Professors Hajime Fujita, Hiroshi Konno 
and Takahiko Yoshida 
for their several useful comments and advices.

\bibliographystyle{plain}

\begin{thebibliography}{10}

\bibitem{andersen1997}
J{\o}rgen~Ellegaard Andersen.
\newblock Geometric quantization of symplectic manifolds with respect to
  reducible non-negative polarizations.
\newblock {\em Communications in mathematical physics}, 183(2):401--421, 1997.

\bibitem{ArnoldAvez1968french}
V.~I. Arnold and A.~Avez.
\newblock {\em Probl\`emes ergodiques de la m\'{e}canique classique}.
\newblock Monographies Internationales de Math\'{e}matiques Modernes, No. 9.
  Gauthier-Villars, \'{E}diteur, Paris, 1967.

\bibitem{BFMN2011}
Thomas Baier, Carlos Florentino, Jos\'{e}~M. Mour\~{a}o, and Jo\~{a}o~P. Nunes.
\newblock Toric {K}\"{a}hler metrics seen from infinity, quantization and
  compact tropical amoebas.
\newblock {\em J. Differential Geom.}, 89(3):411--454, 2011.

\bibitem{BMN2010}
Thomas Baier, Jos\'{e}~M. Mour\~{a}o, and Jo\~{a}o~P. Nunes.
\newblock Quantization of abelian varieties: distributional sections and the
  transition from {K}\"{a}hler to real polarizations.
\newblock {\em J. Funct. Anal.}, 258(10):3388--3412, 2010.

\bibitem{Cheeger-Colding3}
Jeff Cheeger and Tobias~H. Colding.
\newblock On the structure of spaces with {R}icci curvature bounded below.
  {III}.
\newblock {\em J. Differential Geom.}, 54(1):37--74, 2000.

\bibitem{Chern1995}
Shiing-shen Chern.
\newblock {\em Complex manifolds without potential theory (with an appendix on
  the geometry of characteristic classes)}.
\newblock Universitext. Springer-Verlag, New York, second edition, 1995.

\bibitem{Duistermaat1980}
J.~J. Duistermaat.
\newblock On global action-angle coordinates.
\newblock {\em Comm. Pure Appl. Math.}, 33(6):687--706, 1980.

\bibitem{FFY2010}
Hajime Fujita, Mikio Furuta, and Takahiko Yoshida.
\newblock Torus fibrations and localization of index i-polarization and acyclic
  fibrations.
\newblock {\em J. Math. Sci. Univ. Tokyo}, 17(1):1--26, 2010.

\bibitem{Fukaya1987}
Kenji Fukaya.
\newblock Collapsing of {R}iemannian manifolds and eigenvalues of {L}aplace
  operator.
\newblock {\em Invent. Math.}, 87(3):517--547, 1987.

\bibitem{FukayaYamaguchi1994}
Kenji Fukaya and Takao Yamaguchi.
\newblock Isometry groups of singular spaces.
\newblock {\em Math. Z.}, 216(1):31--44, 1994.

\bibitem{HamiltonKonno2014}
Mark~D. Hamilton and Hiroshi Konno.
\newblock Convergence of {K}\"{a}hler to real polarizations on flag manifolds
  via toric degenerations.
\newblock {\em J. Symplectic Geom.}, 12(3):473--509, 2014.

\bibitem{HoffmanWielandt1953}
A.~J. Hoffman and H.~W. Wielandt.
\newblock The variation of the spectrum of a normal matrix.
\newblock {\em Duke Math. J.}, 20:37--39, 1953.

\bibitem{Kasue2011}
Atsushi Kasue.
\newblock Spectral convergence of {R}iemannian vector bundles.
\newblock {\em Sci. Rep. Kanazawa Univ.}, 55:25--49, 2011.

\bibitem{Kubota2016}
Yosuke Kubota.
\newblock The joint spectral flow and localization of the indices of elliptic
  operators.
\newblock {\em Ann. K-Theory}, 1(1):43--83, 2016.

\bibitem{KuwaeShioya2003}
Kazuhiro Kuwae and Takashi Shioya.
\newblock Convergence of spectral structures: a functional analytic theory and
  its applications to spectral geometry.
\newblock {\em Comm. Anal. Geom.}, 11(4):599--673, 2003.

\bibitem{Lott-Dirac2002}
John Lott.
\newblock Collapsing and {D}irac-type operators.
\newblock In {\em Proceedings of the {E}uroconference on {P}artial
  {D}ifferential {E}quations and their {A}pplications to {G}eometry and
  {P}hysics ({C}astelvecchio {P}ascoli, 2000)}, volume~91, pages 175--196,
  2002.

\bibitem{Lott-form2002}
John Lott.
\newblock Collapsing and the differential form {L}aplacian: the case of a
  smooth limit space.
\newblock {\em Duke Math. J.}, 114(2):267--306, 2002.

\bibitem{MarkusMeyer1974}
L.~Markus and K.~R. Meyer.
\newblock {\em Generic {H}amiltonian dynamical systems are neither integrable
  nor ergodic}.
\newblock American Mathematical Society, Providence, R.I., 1974.
\newblock Memoirs of the American Mathematical Society, No. 144.

\bibitem{Tyurin2002}
N.~A. Tyurin.
\newblock Letter to the editors: ``{D}ynamic correspondence in algebraic
  {L}agrangian geometry [{I}zv. {R}oss. {A}kad. {N}auk {S}er. {M}at. {\bf 66}
  (2002), no. 3, 175--196; mr1921813].
\newblock {\em Izv. Ross. Akad. Nauk Ser. Mat.}, 68(3):219--220, 2004.

\bibitem{Tyurin2007}
Nikolai~A. Tyurin.
\newblock Geometric quantization and algebraic {L}agrangian geometry.
\newblock In {\em Surveys in geometry and number theory: reports on
  contemporary {R}ussian mathematics}, volume 338 of {\em London Math. Soc.
  Lecture Note Ser.}, pages 279--318. Cambridge Univ. Press, Cambridge, 2007.

\bibitem{Woodhouse1992}
N.~M.~J. Woodhouse.
\newblock {\em Geometric quantization}.
\newblock Oxford Mathematical Monographs. The Clarendon Press, Oxford
  University Press, New York, second edition, 1992.
\newblock Oxford Science Publications.

\bibitem{yoshida2019adiabatic}
Takahiko Yoshida.
\newblock Adiabatic limits, theta functions, and geometric quantization.
\newblock {\em arXiv preprint arXiv:1904.04076}, 2019.

\end{thebibliography}

\end{document}